\documentclass[10pt,a4paper]{article}

\usepackage[a4paper,left=2.5cm,right=2.5cm,
top=2.5cm,bottom=2.5cm]{geometry}

\usepackage{amsmath,amssymb,amsfonts,amsthm}
\usepackage{mathrsfs}
\usepackage{graphicx}
\usepackage{booktabs}
\usepackage{multirow}
\usepackage{xcolor}
\usepackage{enumerate}
\usepackage{appendix}
\usepackage{hyperref}

\newtheorem{theorem}{Theorem}[section]
\newtheorem{proposition}[theorem]{Proposition}
\newtheorem{lemma}[theorem]{Lemma}

\theoremstyle{definition}
\newtheorem{definition}[theorem]{Definition}
\newtheorem{example}[theorem]{Example}

\theoremstyle{remark}

\numberwithin{equation}{section}

\title{\bfseries
Chen-Ricci and Hineva Inequalities For Riemannian Submersion and Riemannian Maps With Applications
}

\author{
Ravindra Singh
\\[2ex]
$^{a}$Department of Mathematics, Banaras Hindu University, Varanasi 221005, India\\
E-mail: khandelrs@bhu.ac.in\\
ORCID: 0009-0009-1270-3831
}

\date{}

\begin{document}

\maketitle

\begin{abstract}
The Chen--Ricci inequality provides a sharp upper estimate for the Ricci
curvature in terms of the ambient curvature and the squared mean curvature,
whereas the Hineva inequality gives a complementary lower estimate. Although
Chen--Ricci inequalities have been investigated for Riemannian submersions
and Riemannian maps, the corresponding Hineva inequalities have not yet been
systematically studied in these settings. In this paper, we fill this gap by
establishing sharp Hineva inequalities for Riemannian submersions and
Riemannian maps. We also provide alternative and direct proofs of the
Chen--Ricci inequalities by working directly with the Ricci curvature,
rather than proceeding through scalar-curvature identities and optimization
arguments. For Riemannian submersions, we obtain sharp upper and lower
estimates along the vertical and mixed distributions, while for Riemannian
maps we obtain corresponding estimates along the range distribution. The
equality cases are completely characterized, and the simultaneous equality
of the Chen--Ricci and Hineva inequalities is investigated. As applications,
we obtain the corresponding two-sided Ricci-curvature estimates for
Riemannian submersions from real and complex space forms and for Riemannian
maps into real and complex space forms. Several examples are presented to
demonstrate the sharpness of the obtained inequalities.
\end{abstract}


\vspace{3mm}

\noindent{\bf MSC Classification (2020):}
Primary: 53C20; Secondary: 53C25; 53C40.

\section{Introduction}

\label{sec1}

Ricci curvature is one of the fundamental intrinsic invariants of Riemannian
geometry and plays an important role in comparison geometry, geometric
analysis, and the study of the global geometry of Riemannian manifolds. One
of the central themes in submanifold geometry is the establishment of
optimal relationships between intrinsic curvature invariants and extrinsic
invariants. In this direction, Hineva established sharp lower and upper
estimates for the Ricci curvature of submanifolds in terms of the mean
curvature vector and the squared norm of the second fundamental form \cite%
{Hineva,Hineva_84_LNM,Hineva_Beltchev_94_AUSFMI,Hineva_1990}. More recently, Hineva inequality has recently attracted renewed attention and has
been extended to several new geometric settings. Verma et al.~\cite{Verma_2026}
established Chen--Ricci and Hineva inequalities for Kulkarni--Nomizu tensor
fields satisfying an algebraic Gauss equation, with applications to
Riemannian submanifolds of generalized Sasakian space forms. In a subsequent
work, Verma et al.~\cite{Verma_2026_Product} investigated Chen--Ricci and
Hineva inequalities for submanifolds of product generalized Sasakian space
forms. Harry {\em et al.}~\cite{Harry_2026}
studied Hineva inequalities for different classes of submanifolds in
quaternionic space forms. Siddiqui and Mofarreh~\cite{Siddiqui_Mofarreh_2026}
established Hineva inequalities for statistical submanifolds using the
$U$-sectional curvature and further investigated their equality cases and
simultaneous equality with the Chen--Ricci inequality. Moreover, Lone and
Lone~\cite{Lone_2026} obtained a Hineva inequality for submanifolds of
real space forms endowed with a semi-symmetric non-metric connection.
These recent developments demonstrate that Hineva inequalities remain a
highly active and rapidly developing area of research. Analogous
estimates for Ricci curvature were subsequently obtained together with their
equality characterizations \cite{Hineva_84_LNM}. On the other hand, B.-Y.
Chen established several sharp inequalities relating intrinsic and extrinsic
invariants. Among them, the Chen--Ricci inequality gives an optimal upper
estimate for the Ricci curvature of a submanifold of a real space form. More
precisely, if $B$ is an $r$-dimensional $(r\geq 2)$ submanifold of a real
space form of constant sectional curvature $c$, then \cite{Chen_1999} 
\[
{\rm Ric}(X) \leq (r-1)c+\frac{r^{2}}{4}\|H\|^{2}. 
\]
The equality case has been completely characterized, and the Chen--Ricci
inequality has subsequently been generalized in various geometric settings;
see \cite{Chen_Blaga} and the references therein. Thus, the Chen--Ricci and
Hineva inequalities provide complementary upper and lower bounds of Ricci
curvature. Their simultaneous consideration therefore gives a more complete
description of the curvature range determined by the extrinsic geometry.

Riemannian submersions provide a natural framework in which these curvature
inequalities acquire additional geometric structure. Introduced
independently by O'Neill and Gray \cite{Falcitelli_2004,Gray_1967},
Riemannian submersions have been extensively investigated because they
relate the geometry of two Riemannian manifolds and arise naturally in many
areas of differential geometry and mathematical physics \cite%
{BL_1981,IV_1987,Bourguignon_1989,IV_1991,Mustafa_2000}. Let 
\[
\pi:(M_{1},g_{1})\longrightarrow(M_{2},g_{2}) 
\]
be a Riemannian submersion. Then 
\[
TM_{1}={\cal V}\oplus{\cal H}, 
\]
where ${\cal V}=\ker\pi_{*}$ and ${\cal H}=(\ker\pi_{*})^{\perp}$ are the
vertical and horizontal distributions, respectively. The geometry of these
distributions and their interaction is described by O'Neill's fundamental
tensors.

This vertical--horizontal decomposition is precisely what makes curvature
inequalities for Riemannian submersions different from those in classical
submanifold theory. Ricci curvature of the total space contains
contributions from the vertical geometry, the horizontal geometry, and,
importantly, their interaction. In particular, the mixed
vertical--horizontal terms contain information about how the fibres are
related to the horizontal geometry and cannot, in general, be recovered from
the vertical and horizontal contributions separately. Consequently, a
complete Ricci-curvature analysis of a Riemannian submersion should not be
restricted to the individual distributions; the mixed distribution is also
essential. Estimates in the mixed directions therefore provide additional
geometric information about the interaction of the fibres and the base
geometry and contribute to a more complete development of curvature
inequalities for Riemannian submersions.

The importance of this viewpoint becomes particularly clear when one seeks
two-sided estimates. An upper bound of Chen--Ricci inequality alone controls how
large the Ricci curvature can be, whereas a Hineva inequality determines how
small it can be in terms of the same geometric data. Establishing both
bounds along the relevant distributions therefore gives 
\[
\text{Hineva lower bound} \leq {\rm Ric} \leq \text{Chen--Ricci upper bound}%
, 
\]
which provides substantially more information than either estimate
separately. In the submersion setting, obtaining such bounds along the
vertical and mixed distributions is especially relevant because these
directions reflect both the intrinsic geometry of the fibres and their
interaction with the horizontal distribution. Thus, developing two-sided
Ricci-curvature inequalities is a natural step toward a more complete
inequality theory for Riemannian submersions.

Fischer \cite{Fischer_1992} introduced Riemannian maps as a unified
generalization of Riemannian immersions and Riemannian submersions. For a
Riemannian map 
\[
F:(M_{1},g_{1})\longrightarrow(M_{2},g_{2}), 
\]
the tangent bundle of the domain decomposes as 
\[
TM_{1}=\ker F_{*}\oplus(\ker F_{*})^{\perp}, 
\]
while the target space has the corresponding decomposition into the range of 
$F_{*}$ and its orthogonal complement. Hence, Riemannian maps include both
immersions and submersions as special cases and also allow intermediate-rank
mappings. This makes them a natural broader framework for extending
curvature inequalities and for determining which features of the submersion
theory persist for general metric-preserving maps. Several Chen--Ricci inequalities have been established for Riemannian
submersions and Riemannian maps in various ambient geometric settings,
as summarized in Table~\ref{Table_1}.
\begin{table}[ht]
\caption{Chen-Ricci inequalities for Riemannian submersions and Riemannian
maps}
\label{Table_1} \centering
\renewcommand{\arraystretch}{1.15} 
\begin{tabular}{llll}
\hline
{\bf Space Form} & {\bf Geometry} & {\bf Type} & {\bf Reference} \\ \hline
\multirow{4}{*}{Complex} & Riemannian submersion & Lagrangian & \cite%
{Gulbahar_Meric_Kilic} \\ 
& Riemannian submersion & Anti-invariant & \cite{Gunduzalp_Polat_MMN} \\ 
& Riemannian submersion & Slant & \cite{Gunduzalp_Polat_Filomat} \\ 
& Riemannian submersion & Hemi-slant & \cite{Akyol_Demir_Poyraz_Vilcu} \\ 
\hline
Sasakian & Riemannian submersion & Anti-invariant, Slant & \cite%
{Aytimur_Ozgur,Akyol_Poyraz} \\ 
Kenmotsu & Riemannian submersion & Anti-invariant, Slant & \cite%
{Polat_2024,Aquib_Aldayel_Iqbal_Khan} \\ 
Cosymplectic & Riemannian submersion & Anti-invariant, Slant & \cite%
{Aytimur_2023,Poyraz_Akyol} \\ 
Generalized complex & Riemannian submersion & General & \cite{SMM_CR_2025}
\\ 
Generalized Sasakian & Riemannian submersion & General, Slant & \cite%
{SMM_CR_2025,Aquib_2025} \\ \hline
Quaternionic & Riemannian submersion & Anti-invariant & \cite{GGG_25} \\ 
\hline
Real & Riemannian map & General & \cite{LLSV_2022} \\ 
Complex & Riemannian map & General & \cite{LLSV_2022} \\ 
Generalized complex & Riemannian map & General & \cite{SMM_CR_2025} \\ 
Generalized Sasakian & Riemannian map & General & \cite{SMM_CR_2025} \\ 
\hline
\end{tabular}%
\end{table}

Despite these developments, an important gap remains in the theory.
Chen--Ricci inequalities have been investigated for Riemannian submersions
and Riemannian maps, whereas the corresponding Hineva inequalities have not
yet been systematically developed for these classes of maps. Thus, the
existing theory mainly provides upper estimates for Ricci curvature, while
the corresponding sharp lower estimates and their relationship with the
upper estimates remain unexplored. Moreover, the mixed vertical--horizontal
contribution in the submersion theory makes it necessary to consider
directions that are absent in the classical submanifold formulation. Filling
this gap provides not only new inequalities but also a two-sided framework
for understanding how the geometry of the fibres, the horizontal
distribution, and their interaction control Ricci curvature.

In the present work, we employ a direct Ricci-curvature approach to derive
these inequalities. Rather than proceeding indirectly through
scalar-curvature identities and subsequent algebraic optimization, we start
directly from the Ricci curvature in the relevant direction and estimate
its individual curvature contributions. This provides a simple and
transparent way to obtain the Chen--Ricci upper estimate and, at the same
time, naturally leads to the corresponding Hineva lower estimate. The
direct nature of this approach also makes the relationship between the two
bounds and their equality cases more transparent, particularly in the
presence of the horizontal, vertical, mixed, and range distributions
arising in Riemannian submersions and Riemannian maps. Motivated by these
considerations, in this paper we establish Chen--Ricci and Hineva
inequalities for Riemannian submersions and Riemannian maps between
arbitrary Riemannian manifolds. For Riemannian submersions, we derive sharp
upper and lower estimates for Ricci curvature along the vertical and mixed
distributions, while for Riemannian maps, we obtain corresponding
estimates along the range distribution. We completely characterize the
equality cases and investigate the conditions under which the Chen--Ricci
and Hineva bounds are attained simultaneously. As applications, we obtain
the corresponding two-sided inequalities for Riemannian submersions from
real and complex space forms and for Riemannian maps into real and complex
space forms. Several examples are also presented to illustrate the
sharpness of the obtained estimates. Further developments and possible
extensions of the present approach are discussed in Section~\ref{sec_5}.

\section{Preliminaries}

\begin{definition}
{\rm \cite[Tripathi 2025]{Tripathi_25_CMAMS}} The Kulkarni-Nomizu product $%
T_{1}\circledast T_{2}$ of $\left( 0,2\right) $-tensor fields $T_{1}$ and $%
T_{2}$ in a smooth manifold $M$ is a $\left( 0,4\right) $-tensor field
defined by 
\begin{eqnarray}
\left( T_{1}\circledast T_{2}\right) \left( Z_{1},Z_{2},Z_{3},Z_{4}\right)
&=&T_{1}\left( Z_{2},Z_{3}\right) T_{2}\left( Z_{1},Z_{4}\right)
-T_{1}\left( Z_{1},Z_{3}\right) T_{2}\left( Z_{2},Z_{4}\right)  \nonumber \\
&&+T_{2}\left( Z_{2},Z_{3}\right) T_{1}\left( Z_{1},Z_{4}\right)
-T_{2}\left( Z_{1},Z_{3}\right) T_{1}\left( Z_{2},Z_{4}\right)
\label{eq-KN-product}
\end{eqnarray}%
for all vector fields $Z_{1}$, $Z_{2}$, $Z_{3}$, $Z_{4}$ on $M$.
\end{definition}

\begin{definition}
{\rm \cite[Tripathi 2025]{Tripathi_25_CMAMS}} The symmetric product of any
two $\left( 0,2\right) $-tensor fields $T_{1}$ and $T_{2}$ is a $\left(
0,4\right) $-tensor field $T_{1}\circledcirc T_{2}$ defined by 
\begin{equation}
T_{1}\circledcirc T_{2}=T_{1}\circledast T_{2}+T_{2}\circledast T_{1}.
\label{eq-KN-symm-product}
\end{equation}
\end{definition}

Now, let $\left( M,g\right) $ be an $n$-dimensional Riemannian manifold. Let 
$T$ be a Kulkarni-Nomizu tensor field so that it satisfies%
\begin{equation}
T\left( Z_{1},Z_{2},Z_{3},Z_{4}\right) =-T\left(
Z_{2},Z_{1},Z_{3},Z_{4}\right) ,  \label{eq-KN-T-1}
\end{equation}%
\begin{equation}
T\left( Z_{1},Z_{2},Z_{3},Z_{4}\right) =-T\left(
Z_{1},Z_{2},Z_{4},Z_{3}\right) ,  \label{eq-KN-T-2}
\end{equation}%
\begin{equation}
T\left( Z_{1},Z_{2},Z_{3},Z_{4}\right) =T\left(
Z_{3},Z_{4},Z_{1},Z_{2}\right) ,  \label{eq-KN-T-3}
\end{equation}%
\begin{equation}
T\left( Z_{1},Z_{2},Z_{3},Z_{4}\right) +T\left(
Z_{2},Z_{3},Z_{1},Z_{4}\right) +T\left( Z_{3},Z_{1},Z_{2},Z_{4}\right) =0,
\label{eq-KN-T-4}
\end{equation}%
\begin{equation}
T\left( Z_{1},Z_{2},Z_{3},Z_{4}\right) +T\left(
Z_{1},Z_{3},Z_{4},Z_{2}\right) +T\left( Z_{1},Z_{4},Z_{2},Z_{3}\right) =0
\label{eq-KN-T-5}
\end{equation}%
for all vector fields $Z_{1}$, $Z_{2}$, $Z_{3}$ and $Z_{4}$ on $M$. It is
observed that if $T$ satisfies any two of the three conditions (\ref%
{eq-KN-T-1}), (\ref{eq-KN-T-2}), (\ref{eq-KN-T-3}) and any one of the two
conditions (\ref{eq-KN-T-4}), (\ref{eq-KN-T-5}), then it also satisfies the
remanning two relations.

\begin{definition}
{\rm \cite[Kobayashi and Nomizu 1963, p. 209]{Kobayashi_Nomizu_1963}}, {\rm 
\cite[Tripathi 2026, p.~29]{Tripathi_26_WS}} A real space form is a
Riemannian manifold $(M,g)$ whose sectional curvature assumes a constant
value $c$. The corresponding Riemann-Christoffel curvature tensor satisfies
: 
\begin{equation}
R^{M}=\frac{c}{2}\left( g\circledast g\right) .  \label{eq-RSF}
\end{equation}
\end{definition}

\begin{definition}
{\rm \cite[Ogiue 1972]{Ogiue_72_JMSJ}}, {\rm \cite[Tripathi 2026, p.~79]%
{Tripathi_26_WS}} Let $(M,J,g)$ be an almost Hermitian manifold, where $J$
is the almost complex structure and $g$ is the compatible Riemannian metric.
If the almost complex structure is parallel with respect to the Levi-Civita
connection, i.e., $\nabla J=0$, then $(M,J,g) $ is called a K\"{a}hler
manifold. A K\"{a}hler manifold whose holomorphic sectional curvature is
constant and equal to $c$ is known as a complex space form, denoted by $M(c)$%
. The Riemann-Christoffel curvature tensor of $M(c)$ is given by 
\begin{equation}
R^{M}=\frac{c}{8}\left( g \circledast g\right) +\frac{c}{4}\left\{ \frac{1}{2%
}\left( J^{b}\circledast J^{b}\right) -\left( J^{b}\circledcirc J^{b}\right)
\right\} ,  \label{eq-GCSF}
\end{equation}%
where $J^{\flat }$ denotes the $(0,2)$-tensor field associated with the $%
(1,1)$-tensor field $J$, defined by 
\[
J^{\flat }(X,Y)=g(X,JY) 
\]%
for all vector fields $X$, $Y$ on $M$. For each vector field $X$ on $M$, we
express $JX$ as the sum of its horizontal and vertical projections: 
\begin{equation}
JX=PX+QX,  \label{decompose_GCSF_RM}
\end{equation}%
where $PX\in {\cal H}$ and $QX\in {\cal V}$.
\end{definition}

\begin{lemma}
{\rm \cite[cf.~Hineva 2008, Lemma 3.1]{Hineva}} \label{Hineva}Let $B=\left(
b_{ij}\right) _{k\times k}$ be a symmetric matrix with $\left( k\geq
2\right) $, satisfying ${\rm trace}\left( A\right) =b$ and Frobenius norm $%
\left\Vert A\right\Vert =a$. Then%
\begin{equation}
b_{11}\sum_{i=1}^{k}b_{ii}-\sum_{i=1}^{k}\left( b_{1i}\right) ^{2}\geq \frac{%
k-1}{k^{2}}\left( 2b^{2}-ka-\left( k-2\right) \left\vert b\right\vert \sqrt{%
\frac{ka-b^{2}}{k-1}}\right) .  \label{eq-T-1}
\end{equation}%
Moreover, equality is attained if and only if $B$ is a diagonal matrix of
the form%
\[
B={\rm diag}\left( b_{1},b_{2},\ldots ,b_{2}\right) , 
\]%
where%
\[
b_{1}=\frac{b}{k}\mp \frac{k-1}{k}\sqrt{\frac{ka-b^{2}}{k-1}},\ b_{2}=\frac{b%
}{k}\pm \frac{1}{k}\sqrt{\frac{ka-b^{2}}{k-1}}. 
\]
\end{lemma}

\section{Background of Riemannian submersions}

Let $(N_{1},g_{1})$ and $(N_{2},g_{2})$ be Riemannian manifolds of
dimensions $n_{1}$ and $n_{2}$, respectively. A smooth surjective map 
\[
F:(N_{1},g_{1})\rightarrow (N_{2},g_{2}) 
\]
is said to be a {\em Riemannian submersion} if, for every point $p\in N_{1}$%
, the differential 
\[
F_{*p}:T_{p}N_{1}\rightarrow T_{F(p)}N_{2} 
\]
is surjective and acts as an isometry when restricted to the orthogonal
complement of its kernel.

The tangent vectors belonging to $\ker F_{\ast }$ are referred to as {\em %
vertical vectors}, whereas those orthogonal to $\ker F_{\ast }$ are called 
{\em horizontal vectors}. Accordingly, the tangent bundle of $N_{1}$
decomposes orthogonally as 
\[
TN_{1}={\cal V}\oplus {\cal H}, 
\]%
where 
\[
{\cal V}=\ker F_{\ast },\quad {\cal H}=(\ker F_{\ast })^{\perp }. 
\]%
Hence, at each point $p\in N_{1}$, one has the orthogonal decomposition 
\[
T_{p}N_{1}=(\ker F_{\ast p})\oplus (\ker F_{\ast p})^{\perp }, 
\]%
and the restriction 
\[
F_{\ast p}\big|_{(\ker F_{\ast p})^{\perp }}:(\ker F_{\ast p})^{\perp
}\rightarrow {\rm range}F_{\ast p} 
\]%
is a linear isometry. Throughout this paper, we fix an orthonormal basis $%
\{V_{1},\ldots ,V_{\ell },U_{1},\ldots ,U_{s}\}$ of $T_{p}N_{1}$ such that $%
\{V_{1},\ldots ,V_{\ell }\}$ and $\{U_{1},\ldots ,U_{s}\}$ form orthonormal
bases of the vertical distribution ${\cal V}_{p}$ and the horizontal
distribution ${\cal H}_{p}$, respectively. These adapted bases will be used
throughout the subsequent sections.

\subsection{O'Neill tensors}

For a Riemannian submersion, O'Neill \cite{Neill_1996} defined two $(1,2)$%
-type tensor fields {${\cal T}$} and {${\cal A}$} that satisfy 
\[
{\cal A}_{Y_{1}}Y_{2}=-{\cal A}_{Y_{2}}Y_{1}, 
\]%
\[
{\cal T}_{W_{1}}W_{2}={\cal T}_{W_{2}}W_{1}, 
\]%
\[
g_{1}\left( {\cal T}_{W_{1}}Y_{2},Y_{3}\right) =-g_{1}\left( {\cal T}%
_{W_{1}}Y_{3},Y_{2}\right) , 
\]%
and 
\[
g_{1}\left( {\cal A}_{Y_{1}}Y_{2},Y_{3}\right) =-g_{1}\left( {\cal A}%
_{Y_{1}}Y_{3},Y_{2}\right) , 
\]%
where $W_{1},W_{2}\in {\cal V}$ and $Y_{1},Y_{2},Y_{3}\in {\cal H}$.

\subsubsection*{Relations between Riemannian curvature tensors}

Let $R^{N_{1}}$, $R^{N_{2}}$, $R^{{\cal V}}$, and $R^{{\cal H}}$ denote the
Riemannian curvature tensors corresponding to $N_{1}$, $N_{2}$, ${\cal V}$,
and ${\cal H}$, respectively. Then we have the following relations \cite%
{Neill_1996, Falcitelli_2004} 
\begin{eqnarray}
R^{N_{1}}\left( W_{1},W_{2},W_{3},W_{4}\right) &=&R^{\ker F_{\ast }}\left(
W_{1},W_{2},W_{3},W_{4}\right) +g_{1}\left( {\cal T}_{W_{1}}W_{4},{\cal T}%
_{W_{2}}W_{3}\right)  \nonumber \\
&&-g_{1}\left( {\cal T}_{W_{2}}W_{4},{\cal T}_{W_{1}}W_{3}\right) ,
\label{eq-(2.3)}
\end{eqnarray}%
\begin{eqnarray}
R^{N_{1}}\left( Y_{1},Y_{2},Y_{3},Y_{4}\right) &=&R^{\left( \ker F_{\ast
}\right) ^{\perp }}\left( Y_{1},Y_{2},Y_{3},Y_{4}\right) -2g_{1}\left( {\cal %
A}_{Y_{1}}Y_{2},{\cal A}_{Y_{3}}Y_{4}\right)  \nonumber \\
&&+g_{1}\left( {\cal A}_{Y_{2}}Y_{3},{\cal A}_{Y_{1}}Y_{4}\right)
-g_{1}\left( {\cal A}_{Y_{1}}Y_{3},{\cal A}_{Y_{2}}Y_{4}\right) ,
\label{eq-(2.4)}
\end{eqnarray}%
and 
\begin{eqnarray}
R^{N_{1}}\left( Y_{1},W_{1},Y_{2},W_{2}\right) &=&g_{1}\left( \left( \nabla
_{Y_{1}}^{1}{\cal T}\right) \left( W_{1},W_{2}\right) ,Y_{2}\right)
+g_{1}\left( \left( \nabla _{W_{1}}^{1}{\cal A}\right) \left(
Y_{1},Y_{2}\right) ,W_{2}\right)  \nonumber \\
&&-g_{1}\left( {\cal T}_{W_{1}}Y_{1},{\cal T}_{W_{2}}Y_{2}\right)
+g_{1}\left( {\cal A}_{Y_{2}}W_{2},{\cal A}_{Y_{1}}W_{1}\right)
\label{eq-(2.5)}
\end{eqnarray}%
for all $Y_{1},Y_{2},Y_{3},Y_{4}\in {\cal H}$ and $W_{1},W_{2},W_{3},W_{4}%
\in {\cal V}$. Here, $\nabla ^{1}$ is the Levi-Civita connection with
respect to the metric $g_{1}$. In the sequel, we also use curvature-like
notation. Therefore, we fix those as follows. 
\begin{equation}
{\rm Ric}_{{\cal H}}^{{\cal H}}(U_{1})=\sum\limits_{j=2}^{r}R^{{\cal H}%
}\left( U_{1},U_{j},U_{j},U_{1}\right) ,\quad {\rm Ric}_{{\cal H}%
}^{N_{1}}(U_{1})=\sum\limits_{j=2}^{r}R^{N_{1}}\left(
U_{1},U_{j},U_{j},U_{1}\right) ,  \label{eq-RicH}
\end{equation}%
\begin{equation}
{\rm Ric}_{{\cal V}}^{{\cal V}}(V_{1})=\sum\limits_{j=2}^{\ell }R^{{\cal V}%
}\left( V_{1},V_{j},V_{j},V_{1}\right) ,\quad {\rm Ric}_{{\cal V}%
}^{N_{1}}(V_{1})=\sum\limits_{j=2}^{\ell }R^{N_{1}}\left(
V_{1},V_{j},V_{j},V_{1}\right) ,  \label{eq-RicV}
\end{equation}%
\begin{equation}
\sum\limits_{i=1}^{r}{\rm Ric}_{{\cal V}}^{N_{1}}\left( U_{i}\right)
=\sum\limits_{i=1}^{r}\sum\limits_{j=1}^{\ell }R^{N_{1}}\left(
U_{i},V_{j},V_{j},U_{i}\right) .  \label{eq-Ric-UV}
\end{equation}

\subsubsection*{Some notations}

We define some notation as 
\begin{equation}
{\cal T}_{ij}^{{\cal H}^{\alpha }}=g_{1}({\cal T}_{V_{i}}V_{j},U_{\alpha
}),\quad 1\leq i,j\leq \ell ,\ 1\leq \alpha \leq r,  \label{eq-(3.1)}
\end{equation}%
\begin{equation}
{\cal A}_{ij}^{\alpha }=g_{1}({\cal A}_{U_{i}}U_{j},V_{\alpha }),\quad 1\leq
i,j\leq r,\ 1\leq \alpha \leq \ell ,  \label{eq-(3.2)}
\end{equation}%
\begin{equation}
\delta (N)=\sum\limits_{i=1}^{r}\sum\limits_{j=1}^{\ell }\left( \left(
\nabla _{U_{i}}^{1}{\cal T}\right) _{V_{j}}V_{j},U_{i}\right) ,
\label{eq-(3.3)}
\end{equation}%
\begin{equation}
\left\Vert {\cal T}^{{\cal V}}\right\Vert
^{2}=\sum\limits_{i=1}^{r}\sum\limits_{j=1}^{\ell }g_{1}\left( {\cal T}%
_{V_{j}}U_{i},{\cal T}_{V_{j}}U_{i}\right) ,\quad \left\Vert {\cal T}^{{\cal %
H}}\right\Vert ^{2}=\sum\limits_{i,j=1}^{\ell }g_{1}\left( {\cal T}%
_{V_{i}}V_{j},{\cal T}_{V_{i}}V_{j}\right) ,  \label{TVert}
\end{equation}%
\begin{equation}
\left\Vert {\cal A}^{{\cal H}}\right\Vert
^{2}=\sum\limits_{i=1}^{r}\sum\limits_{j=1}^{\ell }g_{1}\left( {\cal A}%
_{U_{i}}V_{j},{\cal A}_{U_{i}}V_{j}\right) ,\quad \left\Vert {\cal A}^{{\cal %
V}}\right\Vert ^{2}=\sum\limits_{i,j=1}^{r}g_{1}\left( {\cal A}_{U_{i}}U_{j},%
{\cal A}_{U_{i}}U_{j}\right) ,  \label{AHor}
\end{equation}%
\begin{equation}
{\rm trace}\left( {\cal T}^{{\cal H}^{\alpha }}\right)
=\sum\limits_{i=1}^{\ell }{\cal T}_{ii}^{{\cal H}^{\alpha }},\ {\rm trace}%
\left( {\cal T}^{{\cal H}^{\alpha }}\right) ^{2}=\sum\limits_{i,j=1}^{\ell
}\left( {\cal T}_{ij}^{{\cal H}^{\alpha }}\right) ^{2},\quad \alpha \in
\left\{ 1,\ldots ,r\right\} ,  \label{traceT}
\end{equation}%
\[
\Vert Q\Vert ^{2}=\sum_{i,j=1}^{\ell }\left( g(QV_{i},V_{j})\right)
^{2},\quad \Vert P^{{\cal V}}\Vert ^{2}=\sum_{i=1}^{\ell
}\sum_{j=1}^{r}\left( g(PV_{i},U_{j})\right) ^{2},\quad \Vert P\Vert
^{2}=\sum_{i,j=1}^{r}\left( g(PU_{i},U_{j})\right) ^{2}. 
\]
The {\it mean curvature vector field }$H${\it \ of the fibers} of $F$ is
defined as \cite{Falcitelli_2004},%
\begin{equation}
H=\frac{1}{\ell }\sum\limits_{i=1}^{\ell }{\cal T}_{V_{i}}V_{i},\quad
\left\Vert H\right\Vert ^{2}=\frac{1}{\ell ^{2}}\sum_{\alpha =1}^{r}\left(
\sum_{j=1}^{\ell }{\cal T}_{jj}^{{\cal H}^{\alpha }}\right) ^{2}.
\label{mean_curv}
\end{equation}

We begin with the following examples of Riemannian submersions, which will
be used to illustrate the sharpness and applicability of the results
established in the subsequent sections.

\begin{example}
Let 
\[
N_{1}=\left\{ \left( x_{1},x_{2},x_{3},x_{4},x_{5},x_{6},x_{7}\right) \in 
{\Bbb R}^{7}\right\} , 
\]%
be endowed with the Riemannian metric 
\[
g_{1}=e^{2x_{7}}\sum_{i=1}^{6}dx_{i}^{2}+dx_{7}^{2}. 
\]%
Further, let 
\[
N_{2}=\left\{ \left( y_{1},y_{2},y_{3},y_{4}\right) \in {\Bbb R}^{4}\right\}
, 
\]%
be equipped with the Riemannian metric 
\[
g_{2}=e^{2y_{4}}\sum_{i=1}^{3}dy_{i}^{2}+dy_{4}^{2}. 
\]%
Consider the smooth map 
\[
F:(N_{1},g_{1})\longrightarrow (N_{2},g_{2}), 
\]%
defined by 
\[
F(x_{1},x_{2},x_{3},x_{4},x_{5},x_{6},x_{7})=\left( \frac{x_{1}+x_{2}}{\sqrt{%
2}},\frac{x_{3}+x_{4}}{\sqrt{2}},\frac{x_{5}+x_{6}}{\sqrt{2}},x_{7}\right) . 
\]%
Let 
\[
\left\{ e_{i}=e^{-x_{7}}\frac{\partial }{\partial x_{i}},\,\,i=1,\ldots
,6,\quad e_{7}=\frac{\partial }{\partial x_{7}}\right\} , 
\]%
and 
\[
\left\{ e_{i}^{\ast }=e^{-y_{4}}\frac{\partial }{\partial y_{i}}%
,\,\,i=1,2,3,\quad e_{4}^{\ast }=\frac{\partial }{\partial y_{4}}\right\} , 
\]%
be orthonormal frames on $(N_{1},g_{1})$ and $(N_{2},g_{2})$, respectively.
The vertical and horizontal distributions of $F$ are given by 
\[
\ker F_{\ast }={\rm span}\left\{
V_{1}=e_{1}-e_{2},\,V_{2}=e_{3}-e_{4},\,V_{3}=e_{5}-e_{6}\right\} , 
\]%
and 
\[
(\ker F_{\ast })^{\perp }={\rm span}\left\{
U_{1}=e_{1}+e_{2},\,U_{2}=e_{3}+e_{4},\,U_{3}=e_{5}+e_{6},\,U_{4}=e_{7}%
\right\} . 
\]%
A straightforward computation shows that 
\[
F_{\ast }(U_{i})=\sqrt{2}\,e_{i}^{\ast },\qquad i=1,2,3, 
\]%
and 
\[
F_{\ast }(U_{4})=e_{4}^{\ast }. 
\]%
Moreover, 
\[
g_{1}(U_{i},U_{j})=g_{2}(F_{\ast }U_{i},F_{\ast }U_{j}),\qquad i,j=1,\ldots
,4. 
\]%
Since ${\rm rank}(F_{\ast })=4$, the map $F$ is a Riemannian submersion. The
only non-vanishing Christoffel symbols of $(M,g_{1})$ are 
\[
\Gamma _{ii}^{7}=-e^{2x_{7}},\qquad i=1,\ldots ,6. 
\]%
Consequently, 
\[
\nabla _{V_{i}}V_{i}=-2\frac{\partial }{\partial x_{7}},\qquad i=1,2,3, 
\]%
while 
\[
\nabla _{V_{i}}V_{j}=0,\qquad i\neq j. 
\]%
Therefore, the only non-zero components of the tensor ${\cal T}$ are 
\[
{\cal T}_{ii}^{{\cal H}^{4}}=-2,\qquad i=1,2,3. 
\]%
Equivalently, 
\[
{\cal T}_{ij}^{{\cal H}^{\alpha }}=\left( 
\begin{array}{ccc}
0 & 0 & 0 \\ 
0 & 0 & 0 \\ 
0 & 0 & 0%
\end{array}%
\right) ,\quad \alpha =1,2,3,\quad {\cal T}_{ij}^{{\cal H}^{4}}=\left( 
\begin{array}{ccc}
-2 & 0 & 0 \\ 
0 & -2 & 0 \\ 
0 & 0 & -2%
\end{array}%
\right) . 
\]%
Since 
\[
{\cal T}_{11}^{{\cal H}^{4}}\neq {\cal T}_{22}^{{\cal H}^{4}}+{\cal T}_{33}^{%
{\cal H}^{4}}. 
\]%
We observe that the equality case of Theorem~\ref{Th-Ric-vert-main} and
Theorem~\ref{Th-Ric-mix-main} is not satisfied for any unit vector in $\ker
F_{\ast }$. On the other hand, the equality conditions of Theorems~\ref%
{Th-Hin-vert-main} and~Theorem \ref{Th-Hin-mix-main} are satisfied for every
unit vector in $\ker F_{\ast }$.
\end{example}

\begin{example}
Consider the smooth map 
\[
F:({\Bbb R}^{6},g_{1})\longrightarrow ({\Bbb R}^{3},g_{2}), 
\]%
defined by 
\[
F(x_{1},x_{2},\ldots ,x_{6})=\left( x_{3},x_{4}\sin \alpha -x_{6}\cos \alpha
,x_{5}\right) ,\qquad 0<\alpha <\frac{\pi }{2}, 
\]%
where the Riemannian metrics on ${\Bbb R}^{6}$ and ${\Bbb R}^{3}$ are given
by 
\[
g_{1}=e^{2x_{3}}(dx_{1})^{2}+e^{2x_{5}}(dx_{2})^{2}+e^{2x_{5}}(dx_{3})^{2}+(dx_{4})^{2}+e^{2x_{3}}(dx_{5})^{2}+(dx_{6})^{2}, 
\]%
and 
\[
g_{2}=e^{2y_{3}}(dy_{1})^{2}+(dy_{2})^{2}+e^{2y_{1}}(dy_{3})^{2}. 
\]%
The vertical and horizontal distributions of $F$ are respectively 
\[
\ker F_{\ast }={\rm span}\left\{ V_{1}=e^{-x_{3}}\frac{\partial }{\partial
x_{1}},\,V_{2}=e^{-x_{5}}\frac{\partial }{\partial x_{2}},\,V_{3}=\cos
\alpha \frac{\partial }{\partial x_{4}}+\sin \alpha \frac{\partial }{%
\partial x_{6}}\right\} , 
\]%
and 
\[
(\ker F_{\ast })^{\perp }={\rm span}\left\{ U_{1}=e^{-x_{5}}\frac{\partial }{%
\partial x_{3}},\,U_{2}=\sin \alpha \frac{\partial }{\partial x_{4}}-\cos
\alpha \frac{\partial }{\partial x_{6}},\,U_{3}=e^{-x_{3}}\frac{\partial }{%
\partial x_{5}}\right\} . 
\]%
A direct computation yields 
\[
F_{\ast }U_{1}=e^{-y_{3}}\frac{\partial }{\partial y_{1}},\qquad F_{\ast
}U_{2}=\frac{\partial }{\partial y_{2}},\qquad F_{\ast }U_{3}=e^{-y_{1}}%
\frac{\partial }{\partial y_{3}}. 
\]%
Moreover, 
\[
g_{1}(U_{i},U_{j})=g_{2}(F_{\ast }U_{i},F_{\ast }U_{j}),\qquad i,j=1,2,3. 
\]%
Since ${\rm rank}(F_{\ast })=3$, the map $F$ is a Riemannian submersion. The
non-vanishing Christoffel symbols of $({\Bbb R}^{6},g_{1})$ are 
\[
\Gamma _{13}^{1}=\Gamma _{31}^{1}=1,\qquad \Gamma _{25}^{2}=\Gamma
_{52}^{2}=1, 
\]%
\[
\Gamma _{35}^{3}=\Gamma _{53}^{3}=1,\qquad \Gamma _{35}^{5}=\Gamma
_{53}^{5}=1, 
\]%
and 
\[
\Gamma _{11}^{3}=-e^{2(x_{3}-x_{5})},\qquad \Gamma
_{22}^{5}=-e^{2(x_{5}-x_{3})}, 
\]%
\[
\Gamma _{55}^{3}=-e^{2(x_{3}-x_{5})},\qquad \Gamma
_{33}^{5}=-e^{2(x_{5}-x_{3})}. 
\]%
Hence, the covariant derivatives of the vertical vector fields are 
\[
\nabla _{V_{1}}V_{1}=\nabla _{e^{-x_{3}}\partial _{x_{1}}}e^{-x_{3}}\partial
_{x_{1}}=-e^{-2x_{5}}\frac{\partial }{\partial x_{3}}, 
\]%
\[
\nabla _{V_{2}}V_{2}=\nabla _{e^{-x_{5}}\partial _{x_{2}}}e^{-x_{5}}\partial
_{x_{2}}=-e^{-2x_{3}}\frac{\partial }{\partial x_{5}}, 
\]%
and 
\[
\nabla _{V_{3}}V_{3}=0, 
\]%
while 
\[
\nabla _{V_{i}}V_{j}=0,\qquad i\neq j. 
\]%
Therefore, the tensor ${\cal T}$ has only two non-zero components, namely, 
\[
({\cal T}^{{\cal H}})_{11}^{1}=-e^{-3x_{5}},\qquad ({\cal T}^{{\cal H}%
})_{22}^{3}=-e^{-3x_{3}}. 
\]%
Consequently,%
\[
{\cal T}_{ij}^{{\cal H}^{1}}=\left( 
\begin{array}{ccc}
-e^{-3x_{5}} & 0 & 0 \\ 
0 & 0 & 0 \\ 
0 & 0 & 0%
\end{array}%
\right) ,\quad {\cal T}_{ij}^{{\cal H}^{2}}=\left( 
\begin{array}{ccc}
0 & 0 & 0 \\ 
0 & 0 & 0 \\ 
0 & 0 & 0%
\end{array}%
\right) ,\quad {\cal T}_{ij}^{{\cal H}^{3}}=\left( 
\begin{array}{ccc}
0 & 0 & 0 \\ 
0 & 0 & 0 \\ 
0 & 0 & -e^{-3x_{3}}%
\end{array}%
\right) . 
\]%
The above computations show that the equality cases in Theorem~\ref%
{Th-Ric-vert-main} and Theorem~\ref{Th-Ric-mix-main} are not attained for
any unit vector in $\ker F_{\ast }$. Moreover, the equality conditions of
Theorems~\ref{Th-Hin-vert-main} and Theorem~\ref{Th-Hin-mix-main} are
likewise not satisfied for any unit vector in $\ker F_{\ast }$.
\end{example}

\begin{example}
Consider the smooth map $F:({\Bbb R}^{4},g_{1})\longrightarrow ({\Bbb R}%
^{2},g_{2})$ is defined by%
\[
F(x_{1},x_{2},x_{3},x_{4})=(x_{3},x_{4}), 
\]%
where 
\[
g_{1}=e^{2f}(dx_{1}^{2}+dx_{2}^{2})+dx_{3}^{2}+dx_{4}^{2}, 
\]%
where $f=ax_{3}+bx_{4}$, $a,b\neq 0$, and $g_{2}$ is the Euclidean metric on 
${\Bbb R}^{2}$. We obtain 
\[
\ker F_{\ast }={\rm span}\left\{ V_{1}=e^{-f}\frac{\partial }{\partial x_{1}}%
,\quad V_{2}=e^{-f}\frac{\partial }{\partial x_{2}}\right\} , 
\]%
\[
(\ker F_{\ast })^{\perp }={\rm span}\left\{ U_{1}=\frac{\partial }{\partial
x_{3}},U_{2}=\frac{\partial }{\partial x_{4}}\right\} . 
\]%
\[
F_{\ast }(U_{1})=\frac{\partial }{\partial y_{1}},\qquad F_{\ast }(U_{2})=%
\frac{\partial }{\partial y_{2}}, 
\]%
it follows that 
\[
g_{1}(U_{i},U_{j})=g_{2}(F_{\ast }U_{i},F_{\ast }U_{j}),\quad i,j\in
\{1,2\}. 
\]%
Since ${\rm rank}(F_{\ast })=2$, the map $F$ is a Riemannian submersion.
Furthermore, the non-zero Christeffel symbols are%
\[
\Gamma _{11}^{3}=-ae^{2f},\ \Gamma _{11}^{4}=-be^{2f},\ \Gamma
_{13}^{1}=\Gamma _{31}^{1}=a,\ \Gamma _{14}^{1}=\Gamma _{41}^{1}=b, 
\]%
\[
\Gamma _{22}^{3}=-ae^{2f},\ \Gamma _{22}^{4}=-be^{2f},\ \Gamma
_{23}^{2}=\Gamma _{32}^{2}=a,\ \Gamma _{24}^{2}=\Gamma _{42}^{2}=b, 
\]%
Now, we obtain 
\[
{\cal T}_{U_{1}}U_{1}={\cal T}_{U_{2}}U_{2}=-aU_{1}-bU_{2},\qquad {\cal T}%
_{U_{1}}U_{2}=0, 
\]%
and therefore 
\[
({\cal T}_{ij}^{{\cal H}^{1}})=\left( 
\begin{array}{cc}
-a & 0 \\ 
0 & -a%
\end{array}%
\right) ,\qquad ({\cal T}_{ij}^{{\cal H}^{2}})=\left( 
\begin{array}{cc}
-b & 0 \\ 
0 & -b%
\end{array}%
\right) . 
\]%
The above computations show that the equality cases in Theorem~\ref%
{Th-Ric-vert-main} and Theorem~\ref{Th-Ric-mix-main} are attained for every
unit vector in $\ker F_{\ast }$. Moreover, the equality conditions of
Theorems~\ref{Th-Hin-vert-main} and Theorem~\ref{Th-Hin-mix-main} are
likewise satisfied for every unit vector in $\ker F_{\ast }$.
\end{example}

\subsection{Chen-Ricci inequality for Riemannian submersion along vertical
distribution}

The following theorem was proved in {\rm \cite[Theorem~3.1, Eq.~(3.11)]%
{SMM_CR_2025}} using scalar curvature. In this paper, we present an
alternative proof by employing Ricci curvature.

\begin{theorem}
{\rm \cite[Theorem~3.1, Eq.~(3.11)]{SMM_CR_2025}} \label{Th-Ric-vert-main}%
Let $F:\left( N_{1},g_{1}\right) \rightarrow \left( N_{2},g_{2}\right) $ be
a Riemannian submersion between two Riemannian manifolds. Then for any unit
vector $V\in {\cal V}_{p}$,%
\[
{\rm Ric}_{{\cal V}}^{{\cal V}}\left( V\right) \geq {\rm Ric}_{{\cal V}%
}^{N_{1}}\left( V\right) -\frac{1}{4}\ell ^{2}\left\Vert H\right\Vert ^{2}. 
\]%
The equality holds if and only if for each $\alpha \in \left\{ 1,\ldots
r\right\} $ and $V=V_{i}$,%
\[
{\cal T}_{ij}^{{\cal H}^{\alpha }}=0,\quad j\in \left\{ 1,\ldots
i-1,i+1,\ldots ,\ell \right\} , 
\]%
\[
{\cal T}_{ii}^{{\cal H}^{\alpha }}={\cal T}_{11}^{{\cal H}^{\alpha }}+\cdots
+{\cal T}_{i-1i-1}^{{\cal H}^{\alpha }}+{\cal T}_{i+1i+1}^{{\cal H}^{\alpha
}}+\cdots +{\cal T}_{\ell \ell }^{\alpha } 
\]%
Equivalently, the equality holds for all unit vector fields if and only if
the fibre of $F$ are

\begin{enumerate}
\item[$(1)$] totally geodesic when $\ell >2;$

\item[$(2)$] totally umbilical when $\ell =2.$
\end{enumerate}
\end{theorem}

\begin{proof}
From (\ref{eq-(2.3)}) and (\ref{eq-RicV}), we obtain%
\begin{equation}
{\rm Ric}_{{\cal V}}^{{\cal V}}\left( V_{1}\right) ={\rm Ric}_{{\cal V}%
}^{N_{1}}\left( V_{1}\right) -\sum_{\alpha =1}^{r}\left( {\cal T}_{11}^{%
{\cal H}^{\alpha }}\sum_{j=2}^{\ell }{\cal T}_{jj}^{{\cal H}^{\alpha
}}-\sum_{j=2}^{\ell }\left( {\cal T}_{1j}^{{\cal H}^{\alpha }}\right)
^{2}\right) .  \label{eq-CR-1}
\end{equation}%
Thus 
\begin{equation}
{\rm Ric}_{{\cal V}}^{{\cal V}}\left( V_{1}\right) \geq {\rm Ric}_{{\cal V}%
}^{N_{1}}\left( V_{1}\right) -\sum_{\alpha =1}^{r}\left( {\cal T}_{11}^{%
{\cal H}^{\alpha }}\sum_{j=2}^{\ell }{\cal T}_{jj}^{{\cal H}^{\alpha
}}\right) .  \label{eq-CR-2}
\end{equation}%
Now, we have the basic inequality%
\begin{equation}
4d_{1}d_{2}\leq \left( d_{1}+d_{2}\right) ^{2},  \label{eq-CR-3}
\end{equation}%
with equality if and only if $d_{1}=d_{2}$. Assume that $d_{1}={\cal T}%
_{11}^{{\cal H}^{\alpha }}$ and $d_{2}=\sum_{j=2}^{\ell }{\cal T}_{jj}^{%
{\cal H}^{\alpha }}$. Then, by using (\ref{eq-CR-2}) and (\ref{eq-CR-3}), we
obtain 
\begin{equation}
{\rm Ric}_{{\cal V}}^{{\cal V}}\left( V_{1}\right) \geq {\rm Ric}_{{\cal V}%
}^{N_{1}}\left( V_{1}\right) -\frac{1}{4}\sum_{\alpha =1}^{r}\left( {\cal T}%
_{11}^{{\cal H}^{\alpha }}+\sum_{j=2}^{\ell }{\cal T}_{jj}^{{\cal H}^{\alpha
}}\right) ^{2}.  \label{eq-CR-4}
\end{equation}%
From (\ref{mean_curv}) and (\ref{eq-CR-4}), we obtain%
\begin{equation}
{\rm Ric}_{{\cal V}}^{{\cal V}}\left( V_{1}\right) \geq {\rm Ric}_{{\cal V}%
}^{N_{1}}\left( V_{1}\right) -\frac{\ell ^{2}}{4}\left\Vert H\right\Vert
^{2}.  \label{eq-CR-5}
\end{equation}%
The equality holds in (\ref{eq-CR-5}) if and only if equality holds in (\ref%
{eq-CR-4}). Consequently, for each $\alpha \in \{1,\ldots ,r\}$, we obtain 
\begin{equation}
{\cal T}_{11}^{{\cal H}^{\alpha }}={\cal T}_{22}^{{\cal H}^{\alpha }}+\cdots
+{\cal T}_{\ell \ell }^{{\cal H}^{\alpha }},\quad \quad {\cal T}_{1j}^{{\cal %
H}^{\alpha }}=0,\quad j\in \{2,\ldots ,\ell \}.  \label{eq-CR-6}
\end{equation}%
Assuming that equality holds for all unit vector field $V_{i}$, $i\in
\{1,\ldots ,\ell \}$. Then from (\ref{eq-CR-6}), for each $\alpha \in
\{1,\ldots ,r\}$, we obtain 
\begin{equation}
{\cal T}_{ij}^{{\cal H}^{\alpha }}=0,\quad i\neq j,  \label{eq-CR-7}
\end{equation}%
\begin{equation}
2T_{ii}^{{\cal H}^{\alpha }}={\cal T}_{11}^{{\cal H}^{\alpha }}+{\cal T}%
_{22}^{{\cal H}^{\alpha }}+\cdots +{\cal T}_{\ell \ell }^{{\cal H}^{\alpha
}},\quad i\in \{1,\ldots \ell \}.  \label{eq-CR-8}
\end{equation}%
From (\ref{eq-CR-8}), we obtain%
\begin{equation}
\left( \ell -2\right) ({\cal T}_{11}^{{\cal H}^{\alpha }}+{\cal T}_{22}^{%
{\cal H}^{\alpha }}+\cdots +{\cal T}_{\ell \ell }^{{\cal H}^{\alpha }})=0.
\label{eq-CR-9}
\end{equation}%
From (\ref{eq-CR-9}), we obtain either $\ell =2$ or ${\cal T}_{11}^{{\cal H}%
^{\alpha }}+{\cal T}_{22}^{{\cal H}^{\alpha }}+\cdots +{\cal T}_{\ell \ell
}^{{\cal H}^{\alpha }}=0$. If ${\cal T}_{11}^{{\cal H}^{\alpha }}+{\cal T}%
_{22}^{{\cal H}^{\alpha }}+\cdots +{\cal T}_{\ell \ell }^{{\cal H}^{\alpha
}}=0$, then from (\ref{eq-CR-8}), we obtain ${\cal T}_{ii}^{{\cal H}^{\alpha
}}=0$ for all $\alpha \in \{1,\ldots r\}$ and $i\in \{1,\ldots ,\ell \}$. In
view of (\ref{eq-CR-7}), we obtain ${\cal T}_{ij}^{{\cal H}^{\alpha }}=0$,
for all $\alpha \in \{1,\ldots r\}$ and $i,j\in \{1,\ldots ,\ell \}$. Which
shows that fibres of $F$ are totally geodesic. If $\ell =2$, then from (\ref%
{eq-CR-7}) and (\ref{eq-CR-8}), for each $\alpha \in \{1,\ldots ,r\}$, we
obtain 
\begin{equation}
{\cal T}_{11}^{{\cal H}^{\alpha }}={\cal T}_{22}^{{\cal H}^{\alpha }},\quad 
{\cal T}_{12}^{{\cal H}^{\alpha }}=0.  \label{eq-CR-10}
\end{equation}%
From (\ref{eq-CR-10}), it is straightforward that the fibers of $F$ are
totally umbilical.
\end{proof}

\subsection{Chen-Ricci inequality for Riemannian submersion along mixed
distribution}

The following theorem was proved in {\rm \cite[Theorem~3.1, Eq.~(3.13)]%
{SMM_CR_2025}} using scalar curvature. In this paper, we present an
alternative proof by employing Ricci curvature.

\begin{theorem}
{\rm \cite[Theorem~3.1, Eq.~(3.13)]{SMM_CR_2025}} \label{Th-Ric-mix-main}Let 
$F:(N_{1},g_{1})\rightarrow (N_{2},g_{2})$ be a Riemannian submersion
between two Riemannian manifolds. Then%
\begin{align}
& {\rm Ric}_{{\cal V}}^{N_{1}}\left( V_{1}\right) +{\rm Ric}_{{\cal H}%
}^{N_{1}}\left( U_{1}\right) +\sum\limits_{i=1}^{r}{\rm Ric}_{{\cal V}%
}^{N_{1}}\left( U_{i}\right)  \nonumber \\
& \leq {\rm Ric}_{{\cal V}}^{{\cal V}}\left( V_{1}\right) +{\rm Ric}_{{\cal H%
}}^{{\cal H}}\left( U_{1}\right) +\frac{\ell ^{2}}{4}\left\Vert H\right\Vert
^{2}  \nonumber \\
& +3\sum\limits_{\alpha =1}^{\ell }\sum\limits_{j=2}^{r}\left( {\cal A}%
_{1j}^{\alpha }\right) ^{2}-\delta \left( N\right) +\left\Vert {\cal T}^{%
{\cal V}}\right\Vert ^{2}-\left\Vert {\cal A}^{{\cal H}}\right\Vert ^{2}.
\label{eq-main-result}
\end{align}%
The equality case follows {\rm Theorem~\ref{Th-Ric-vert-main}}.
\end{theorem}

\begin{proof} From \cite[Equation (4.11)]{Aytimur_Ozgur}, we have
the scalar curvature $\tau ^{N_{1}}$ of $N_{1}$%
\begin{equation}
\tau ^{N_{1}}=\sum_{1\leq i<j\leq \ell }R^{N_{1}}\left(
V_{i},V_{j},V_{j},V_{i}\right) +\sum_{1\leq i<j\leq r}R^{N_{1}}\left(
U_{i},U_{j},U_{j},U_{i}\right) +\sum_{i=1}^{r}\sum_{j=1}^{\ell
}R^{N_{1}}\left( U_{i},V_{j},V_{j},U_{i}\right) .  \label{eq-mx-1}
\end{equation}%
We can rewrite (\ref{eq-mx-1}) as%
\begin{eqnarray}
&&\tau ^{N_{1}}-\sum_{2\leq i<j\leq \ell }R^{N_{1}}\left(
V_{i},V_{j},V_{j},V_{i}\right) -\sum_{2\leq i<j\leq r}R^{N_{1}}\left(
U_{i},U_{j},U_{j},U_{i}\right)  \nonumber \\
&=&\sum_{j=2}^{\ell }R^{N_{1}}\left( V_{1},V_{j},V_{j},V_{1}\right)
+\sum_{j=2}^{r}R^{N_{1}}\left( U_{1},U_{j},U_{j},U_{1}\right)
+\sum_{i=1}^{r}\sum_{j=1}^{\ell }R^{N_{1}}\left(
U_{i},V_{j},V_{j},U_{i}\right) .  \label{eq-mx-2}
\end{eqnarray}%
From (\ref{eq-(2.3)}), (\ref{eq-(2.4)}), (\ref{eq-(2.5)}), (\ref{eq-(3.1)}),
(\ref{eq-(3.2)}), (\ref{eq-(3.3)}), (\ref{eq-RicH}) and (\ref{eq-RicV}) and (%
\ref{eq-mx-2}), we obtain%
\begin{eqnarray}
&&\tau ^{N_{1}}-\sum_{2\leq i<j\leq \ell }R^{N_{1}}\left(
V_{i},V_{j},V_{j},V_{i}\right) -\sum_{2\leq i<j\leq r}R^{N_{1}}\left(
U_{i},U_{j},U_{j},U_{i}\right)  \nonumber \\
&=&{\rm Ric}_{{\cal V}}^{{\cal V}}\left( V_{1}\right) +\sum_{\alpha
=1}^{r}\left( {\cal T}_{11}^{{\cal H}^{\alpha }}\sum_{j=2}^{\ell }{\cal T}%
_{jj}^{{\cal H}^{\alpha }}-\sum_{j=2}^{\ell }\left( {\cal T}_{1j}^{{\cal H}%
^{\alpha }}\right) ^{2}\right) +{\rm Ric}_{{\cal H}}^{{\cal H}}\left(
U_{1}\right)  \nonumber \\
&&+3\sum\limits_{\alpha =1}^{\ell }\sum\limits_{j=2}^{r}\left( {\cal A}%
_{1j}^{\alpha }\right) ^{2}-\delta \left( N\right) +\left\Vert {\cal T}^{%
{\cal V}}\right\Vert ^{2}-\left\Vert {\cal A}^{{\cal H}}\right\Vert ^{2}.
\label{eq-mx-3}
\end{eqnarray}%
Using (\ref{eq-RicH}), (\ref{eq-RicV}), (\ref{eq-mx-1}) in (\ref{eq-mx-3}),
we obtain 
\begin{eqnarray}
&&{\rm Ric}_{{\cal V}}^{N_{1}}\left( V_{1}\right) +{\rm Ric}_{{\cal H}%
}^{N_{1}}\left( U_{1}\right) +\sum\limits_{i=1}^{r}{\rm Ric}_{{\cal V}%
}^{N_{1}}\left( U_{i}\right)  \nonumber \\
&=&{\rm Ric}_{{\cal V}}^{{\cal V}}\left( V_{1}\right) +\sum_{\alpha
=1}^{r}\left( {\cal T}_{11}^{{\cal H}^{\alpha }}\sum_{j=2}^{\ell }{\cal T}%
_{jj}^{{\cal H}^{\alpha }}-\sum_{j=2}^{\ell }\left( {\cal T}_{1j}^{{\cal H}%
^{\alpha }}\right) ^{2}\right) +{\rm Ric}_{{\cal H}}^{{\cal H}}\left(
U_{1}\right)  \nonumber \\
&&+3\sum\limits_{\alpha =1}^{\ell }\sum\limits_{j=2}^{r}\left( {\cal A}%
_{1j}^{\alpha }\right) ^{2}-\delta \left( N\right) +\left\Vert {\cal T}^{%
{\cal V}}\right\Vert ^{2}-\left\Vert {\cal A}^{{\cal H}}\right\Vert ^{2}.
\label{eq-mx-4}
\end{eqnarray}%
Thus, we obtain%
\begin{eqnarray}
&&{\rm Ric}_{{\cal V}}^{N_{1}}\left( V_{1}\right) +{\rm Ric}_{{\cal H}%
}^{N_{1}}\left( U_{1}\right) +\sum\limits_{i=1}^{r}{\rm Ric}_{{\cal V}%
}^{N_{1}}\left( U_{i}\right)  \nonumber \\
&\leq &{\rm Ric}_{{\cal V}}^{{\cal V}}\left( V_{1}\right) +\sum_{\alpha
=1}^{r}\left( {\cal T}_{11}^{{\cal H}^{\alpha }}\sum_{j=2}^{\ell }{\cal T}%
_{jj}^{{\cal H}^{\alpha }}\right) +{\rm Ric}_{{\cal H}}^{{\cal H}}\left(
U_{1}\right)  \nonumber \\
&&+3\sum\limits_{\alpha =1}^{\ell }\sum\limits_{j=2}^{r}\left( {\cal A}%
_{1j}^{\alpha }\right) ^{2}-\delta \left( N\right) +\left\Vert {\cal T}^{%
{\cal V}}\right\Vert ^{2}-\left\Vert {\cal A}^{{\cal H}}\right\Vert ^{2}.
\label{eq-mx-5}
\end{eqnarray}%
Now, we have the basic inequality%
\begin{equation}
4d_{1}d_{2}\leq \left( d_{1}+d_{2}\right) ^{2},  \label{eq-mix-6}
\end{equation}%
with equality if and only if $d_{1}=d_{2}$. Assume $d_{1}={\cal T}_{11}^{%
{\cal H}^{\alpha }}$ and $d_{2}=\sum_{j=2}^{\ell }{\cal T}_{jj}^{{\cal H}%
^{\alpha }}$. From (\ref{eq-mx-5}) and (\ref{eq-mix-6}), we obtain%
\begin{align}
& {\rm Ric}_{{\cal V}}^{N_{1}}\left( V_{1}\right) +{\rm Ric}_{{\cal H}%
}^{N_{1}}\left( U_{1}\right) +\sum\limits_{i=1}^{r}{\rm Ric}_{{\cal V}%
}^{N_{1}}\left( U_{i}\right)  \nonumber \\
& \leq {\rm Ric}_{{\cal V}}^{{\cal V}}\left( V_{1}\right) +{\rm Ric}_{{\cal H%
}}^{{\cal H}}\left( U_{1}\right) +\frac{1}{4}\sum_{\alpha =1}^{r}\left( 
{\cal T}_{11}^{{\cal H}^{\alpha }}+\sum_{j=2}^{\ell }{\cal T}_{jj}^{{\cal H}%
^{\alpha }}\right) ^{2}  \nonumber \\
& +3\sum\limits_{\alpha =1}^{\ell }\sum\limits_{j=2}^{r}\left( {\cal A}%
_{1j}^{\alpha }\right) ^{2}-\delta \left( N\right) +\left\Vert {\cal T}^{%
{\cal V}}\right\Vert ^{2}-\left\Vert {\cal A}^{{\cal H}}\right\Vert ^{2}.
\label{eq-mx-7}
\end{align}%
From (\ref{mean_curv}) and (\ref{eq-mx-7}), we obtain%
\begin{align}
& {\rm Ric}_{{\cal V}}^{N_{1}}\left( V_{1}\right) +{\rm Ric}_{{\cal H}%
}^{N_{1}}\left( U_{1}\right) +\sum\limits_{i=1}^{r}{\rm Ric}_{{\cal V}%
}^{N_{1}}\left( U_{i}\right)  \nonumber \\
& \leq {\rm Ric}_{{\cal V}}^{{\cal V}}\left( V_{1}\right) +{\rm Ric}_{{\cal H%
}}^{{\cal H}}\left( U_{1}\right) +\frac{\ell ^{2}}{4}\left\Vert H\right\Vert
^{2}  \nonumber \\
& +3\sum\limits_{\alpha =1}^{\ell }\sum\limits_{j=2}^{r}\left( {\cal A}%
_{1j}^{\alpha }\right) ^{2}-\delta \left( N\right) +\left\Vert {\cal T}^{%
{\cal V}}\right\Vert ^{2}-\left\Vert {\cal A}^{{\cal H}}\right\Vert ^{2}.
\label{eq-mx-8}
\end{align}%
equality holds in (\ref{eq-mx-8}) if and only if equality holds in (\ref%
{eq-mx-7}). Then, for each $\alpha \in \{1,\ldots ,r\}$, we obtain 
\[
{\cal T}_{11}^{{\cal H}^{\alpha }}={\cal T}_{22}^{{\cal H}^{\alpha }}+\cdots
+{\cal T}_{\ell \ell }^{{\cal H}^{\alpha }},\quad \quad {\cal T}_{1j}^{{\cal %
H}^{\alpha }}=0,\quad j\in \{2,\ldots ,\ell \}. 
\]
\end{proof}

\subsection{Hineva inequality along the vertical distribution}

We begin with the following theorem, which establishes the Hineva inequality
providing an upper bound for the Ricci curvature of a Riemannian submersion
along vertical distributions.

\begin{theorem}
\label{Th-Hin-vert-main}Let $F:(N_{1},g_{1})\rightarrow (N_{2},g_{2})$ be a
Riemannian submersion between two Riemannian manifolds. Then for any unit
vector $V$, 
\begin{eqnarray}
{\rm Ric}_{{\cal V}}^{{\cal V}}\left( V\right) &\leq &{\rm Ric}_{{\cal V}%
}^{N_{1}}\left( V\right) -\frac{\ell -1}{\ell }\left( 2\ell \left\Vert
H\right\Vert ^{2}-\left\Vert {\cal T}^{{\cal H}}\right\Vert ^{2}\right. 
\nonumber \\
&&\left. -\left( \ell -2\right) \sqrt{\frac{\ell \left\Vert H\right\Vert
^{2}\left( \left\Vert {\cal T}^{{\cal H}}\right\Vert ^{2}-\ell \left\Vert
H\right\Vert ^{2}\right) }{\ell -1}}\right) .  \label{eq-GINQ-RS-V}
\end{eqnarray}%
The equality holds in {\rm (\ref{eq-GINQ-RS-V})} if and only if the matrices 
$\left( {\cal T}_{ij}^{{\cal H}^{\alpha }}\right) $ takes the following
forms:%
\begin{equation}
{\cal T}^{{\cal H}^{\alpha }}={\rm diag}\left( \mu _{\alpha },\ldots ,%
\underset{i^{theorem}\ {\rm position}}{\underbrace{\lambda _{\alpha }}},\ldots
,\mu _{\alpha }\right) ,\quad V=V_{i},\quad i=1,\ldots ,\ell ,\quad \alpha
\in \left\{ 1,\ldots ,r\right\} ,  \label{eq-Ver-equality}
\end{equation}%
and%
\[
\beta _{1}\left\vert \lambda _{\alpha }-\mu _{\alpha }\right\vert +\beta
_{2}\left( \lambda _{\alpha }+\left( \ell -1\right) \mu _{\alpha }\right)
=0,\quad \alpha \in \left\{ 1,\ldots ,r\right\} , 
\]%
where $\beta _{1}$ and $\beta _{2}$ are scalars but not both zero, $\lambda
_{\alpha }$ and $\mu _{\alpha }$ are two eigen values of each one of the
matrices $\left( {\cal T}_{ij}^{{\cal H}^{\alpha }}\right) $ of the
symmetric $\left( 1,2\right) $ tensor field ${\cal T}^{{\cal H}}$ such that%
\begin{eqnarray*}
\ \lambda _{\alpha } &=&\frac{{\rm trace}\left( {\cal T}^{{\cal H}^{\alpha
}}\right) }{\ell }\mp \frac{\ell -1}{\ell }\sqrt{\frac{\ell {\rm trace}%
\left( {\cal T}^{{\cal H}^{\alpha }}\right) ^{2}-\left( {\rm trace}\left( 
{\cal T}^{{\cal H}^{\alpha }}\right) \right) ^{2}}{\ell -1}}, \\
\mu _{\alpha } &=&\frac{{\rm trace}\left( {\cal T}^{{\cal H}^{\alpha
}}\right) }{\ell }\pm \frac{1}{\ell }\sqrt{\frac{\ell {\rm trace}\left( 
{\cal T}^{{\cal H}^{\alpha }}\right) ^{2}-\left( {\rm trace}\left( {\cal T}^{%
{\cal H}^{\alpha }}\right) \right) ^{2}}{\ell -1}}.
\end{eqnarray*}%
The equality conditions can be interpreted as follows. We observe that $%
g_{1}(v_{i},{\cal T}_{v_{i}}U_{\alpha })\neq g_{1}(v_{1},{\cal T}%
_{v_{1}}U_{\alpha })=\cdots =g_{1}(v_{i-1},{\cal T}_{v_{i-1}}U_{\alpha
})=\cdots =g_{1}(v_{\ell -1},{\cal T}_{v_{\ell -1}}U_{\alpha
})=g_{1}(v_{\ell },{\cal T}_{v_{\ell }}U_{\alpha })$ with respect to all
horizontal directions $(U_{\alpha },\text{where}~\alpha \in \{1,\dots ,r\})$%
. Equivalently, there exist $r$ mutually orthogonal horizontal unit vector
fields such that the shape operator with respect to all directions has an
eigenvalue of multiplicity $(\ell -1)$ and that for each $U_{\alpha }$ the
distinguished eigendirections are the same (namely $v_{i}$). Hence, the
leaves of vertical space (called fibers of $F$) are quasi-umbilical {\rm 
\cite{DHV_2008}}.
\end{theorem}

\begin{proof} From (\ref{eq-(2.3)}) and (\ref{eq-RicV}), we obtain 
\begin{equation}
{\rm Ric}_{{\cal V}}^{{\cal V}}\left( V_{1}\right) ={\rm Ric}_{{\cal V}%
}^{N_{1}}\left( V_{1}\right) -\sum_{\alpha =1}^{r}\left( {\cal T}_{11}^{%
{\cal H}^{\alpha }}\sum_{j=1}^{\ell }{\cal T}_{jj}^{{\cal H}^{\alpha
}}-\sum_{j=1}^{\ell }\left( {\cal T}_{1j}^{{\cal H}^{\alpha }}\right)
^{2}\right) .  \label{eq-ver-1}
\end{equation}%
Using Lemma \ref{Hineva}, we obtain 
\begin{eqnarray}
{\cal T}_{11}^{{\cal H}^{\alpha }}\sum_{j=1}^{\ell }{\cal T}_{jj}^{{\cal H}%
^{\alpha }}-\sum_{j=1}^{\ell }\left( {\cal T}_{1j}^{{\cal H}^{\alpha
}}\right) ^{2} &\geq &\frac{\ell -1}{\ell ^{2}}\left( 2\left(
\sum_{j=1}^{\ell }{\cal T}_{jj}^{{\cal H}^{\alpha }}\right) ^{2}-\ell
\sum_{i,j=1}^{\ell }\left( {\cal T}_{ij}^{{\cal H}^{\alpha }}\right)
^{2}\right. -\left( \ell -2\right)   \nonumber \\
&&\left. \left\vert \sum_{j=1}^{\ell }{\cal T}_{jj}^{{\cal H}^{\alpha
}}\right\vert \sqrt{\frac{\ell \sum_{i,j=1}^{\ell }\left( {\cal T}_{ij}^{%
{\cal H}^{\alpha }}\right) ^{2}-\left( \sum_{j=1}^{\ell }{\cal T}_{jj}^{%
{\cal H}^{\alpha }}\right) ^{2}}{\ell -1}}\right\} .  \label{eq-ver-2}
\end{eqnarray}%
From (\ref{eq-ver-1}) and (\ref{eq-ver-2}), we obtain

\begin{eqnarray}
{\rm Ric}_{{\cal V}}^{{\cal V}}\left( V_{1}\right) &\leq &{\rm Ric}_{{\cal V}%
}^{N_{1}}\left( V_{1}\right) -\frac{\ell -1}{\ell ^{2}}\left( 2\sum_{\alpha
=1}^{r}\left( \sum_{j=1}^{\ell }{\cal T}_{jj}^{{\cal H}^{\alpha }}\right)
^{2}-\ell \sum_{\alpha =1}^{r}\sum_{i,j=1}^{\ell }\left( {\cal T}_{ij}^{%
{\cal H}^{\alpha }}\right) ^{2}\right.  \nonumber \\
&&\left. -\left( \ell -2\right) \sum_{\alpha =1}^{r}\left\vert
\sum_{j=1}^{\ell }{\cal T}_{jj}^{{\cal H}^{\alpha }}\right\vert \sqrt{\frac{%
\ell \sum_{i,j=1}^{\ell }\left( {\cal T}_{ij}^{{\cal H}^{\alpha }}\right)
^{2}-\left( \sum_{j=1}^{\ell }{\cal T}_{jj}^{{\cal H}^{\alpha }}\right) ^{2}%
}{\ell -1}}\right\} .  \label{eq-vert-3}
\end{eqnarray}%
By using Cauchy Squartz inequality, we have%
\begin{equation}
\sum_{\alpha =1}^{r}\left\vert \sum_{j=1}^{\ell }{\cal T}_{jj}^{{\cal H}%
^{\alpha }}\right\vert \sqrt{\frac{\ell \sum_{i,j=1}^{\ell }\left( {\cal T}%
_{ij}^{{\cal H}^{\alpha }}\right) ^{2}-\left( \sum_{j=1}^{\ell }{\cal T}%
_{jj}^{{\cal H}^{\alpha }}\right) ^{2}}{\ell -1}}\leq \ell \left\Vert
H\right\Vert \sqrt{\frac{\ell \left\Vert {\cal T}^{{\cal H}}\right\Vert
^{2}-\ell ^{2}\left\Vert H\right\Vert ^{2}}{\ell -1}}.  \label{Cauchy vert}
\end{equation}%
From (\ref{eq-vert-3}), (\ref{Cauchy vert}) and (\ref{mean_curv}), we obtain%
\begin{eqnarray*}
{\rm Ric}_{{\cal V}}^{{\cal V}}\left( V_{1}\right) &\leq &{\rm Ric}_{{\cal V}%
}^{N_{1}}\left( V_{1}\right) -\frac{\ell -1}{\ell }\left( 2\ell \left\Vert
H\right\Vert ^{2}-\left\Vert {\cal T}^{{\cal H}}\right\Vert ^{2}\right. \\
&&\left. -\left( \ell -2\right) \sqrt{\frac{\ell \left\Vert H\right\Vert
^{2}\left( \left\Vert {\cal T}^{{\cal H}}\right\Vert ^{2}-\ell \left\Vert
H\right\Vert ^{2}\right) }{\ell -1}}\right) .
\end{eqnarray*}%
Now, we establish the equality case. Equality in the desired inequality is
attained if and only if equality holds simultaneously in (\ref{eq-ver-2})
and (\ref{Cauchy vert}). If equality holds in (\ref{eq-ver-2}), then, for
each $\alpha \in \{1,\ldots ,r\}$, the matrix $({\cal T}_{ij}^{{\cal H}%
^{\alpha }})$ has the form 
\[
{\cal T}_{ij}^{{\cal H}^{\alpha }}={\rm diag}(\lambda _{\alpha },\mu
_{\alpha },\ldots ,\mu _{\alpha }). 
\]%
On the other hand, equality in (\ref{Cauchy vert}) holds if and only if the
vectors involved in the Cauchy--Schwarz inequality are linearly dependent.
Consequently, there exist scalars $\beta _{1},\beta _{2}$, not both zero,
such that 
\[
\beta _{1}\left\vert \sum_{j=1}^{\ell }{\cal T}_{jj}^{{\cal H}^{\alpha
}}\right\vert +\beta _{2}\sqrt{\frac{\ell \sum_{i,j=1}^{\ell }({\cal T}%
_{ij}^{{\cal H}^{\alpha }})^{2}-\left( \sum_{j=1}^{\ell }{\cal T}_{jj}^{%
{\cal H}^{\alpha }}\right) ^{2}}{\ell -1}}=0,\qquad \alpha \in \{1,\ldots
,r\}. 
\]%
which implies 
\[
\beta _{1}|\lambda _{\alpha }-\mu _{\alpha }|+\beta _{2}\left( \lambda
_{\alpha }+(\ell -1)\mu _{\alpha }\right) =0,\qquad \alpha \in \{1,\ldots
,r\}. 
\]%
This completes the proof.
\end{proof}

\subsection{Hineva inequality along the mixed distributions}

We begin with the following theorem, which establishes the Hineva inequality
providing an upper bound for the Ricci curvature of a Riemannian submersion
along mixed distributions.

\begin{theorem}
\label{Th-Hin-mix-main}Let $F:(N_{1},g_{1})\rightarrow (N_{2},g_{2})$ be a
Riemannian submersion between two Riemannian manifolds. Then%
\begin{eqnarray}
&&{\rm Ric}_{{\cal V}}^{N_{1}}\left( V_{1}\right) +{\rm Ric}_{{\cal H}%
}^{N_{1}}\left( U_{1}\right) +\sum\limits_{i=1}^{r}{\rm Ric}_{{\cal V}%
}^{N_{1}}\left( U_{i}\right)  \nonumber \\
&\geq &{\rm Ric}_{{\cal V}}^{{\cal V}}\left( V_{1}\right) +{\rm Ric}_{{\cal H%
}}^{{\cal H}}\left( U_{1}\right) +\frac{\ell -1}{\ell }\left( 2\ell
\left\Vert H\right\Vert ^{2}-\left\Vert {\cal T}^{{\cal H}}\right\Vert
^{2}\right.  \nonumber \\
&&\left. -\left( \ell -2\right) \sqrt{\frac{\ell \left\Vert H\right\Vert
^{2}\left( \left\Vert {\cal T}^{{\cal H}}\right\Vert ^{2}-\ell \left\Vert
H\right\Vert ^{2}\right) }{\ell -1}}\right)  \nonumber \\
&&+3\sum\limits_{\alpha =1}^{\ell }\sum\limits_{j=2}^{r}\left( {\cal A}%
_{1j}^{\alpha }\right) ^{2}-\delta \left( N\right) +\left\Vert {\cal T}^{%
{\cal V}}\right\Vert ^{2}-\left\Vert {\cal A}^{{\cal H}}\right\Vert ^{2}.
\label{eq-GINQ-RS-HV}
\end{eqnarray}%
The equality case follows {\rm Theorem~\ref{Th-Hin-vert-main}}.
\end{theorem}

\begin{proof}From \cite[Equation (4.11)]{Aytimur_Ozgur}, we have
the scalar curvature $\tau ^{N_{1}}$ of $N_{1}$%
\begin{equation}
\tau ^{N_{1}}=\sum_{1\leq i<j\leq \ell }R^{N_{1}}\left(
V_{i},V_{j},V_{j},V_{i}\right) +\sum_{1\leq i<j\leq r}R^{N_{1}}\left(
U_{i},U_{j},U_{j},U_{i}\right) +\sum\limits_{i=1}^{r}{\rm Ric}_{{\cal V}%
}^{N_{1}}\left( U_{i}\right) ,  \label{eq-mx-1H}
\end{equation}%
equation (\ref{eq-mx-1H}) can be written as%
\begin{eqnarray}
&&\tau ^{N_{1}}-\sum_{2\leq i<j\leq \ell }R^{N_{1}}\left(
V_{i},V_{j},V_{j},V_{i}\right) -\sum_{2\leq i<j\leq r}R^{N_{1}}\left(
U_{i},U_{j},U_{j},U_{i}\right)   \nonumber \\
&=&\sum_{j=2}^{\ell }R^{N_{1}}\left( V_{1},V_{j},V_{j},V_{1}\right)
+\sum_{j=2}^{r}R^{N_{1}}\left( U_{1},U_{j},U_{j},U_{1}\right)
+\sum\limits_{i=1}^{r}{\rm Ric}_{{\cal V}}^{N_{1}}\left( U_{i}\right) .
\label{eq-mx-2H}
\end{eqnarray}%
From (\ref{eq-(2.3)}), (\ref{eq-(2.4)}), (\ref{eq-(2.5)}), (\ref{eq-(3.1)}),
(\ref{eq-(3.2)}), (\ref{eq-(3.3)}), (\ref{eq-RicH}), (\ref{eq-RicV}) and (%
\ref{eq-mx-2H}), we obtain%
\begin{eqnarray}
&&\tau ^{N_{1}}-\sum_{2\leq i<j\leq \ell }R^{N_{1}}\left(
V_{i},V_{j},V_{j},V_{i}\right) -\sum_{2\leq i<j\leq r}R^{N_{1}}\left(
U_{i},U_{j},U_{j},U_{i}\right)   \nonumber \\
&=&{\rm Ric}_{{\cal V}}^{{\cal V}}\left( V_{1}\right) +\sum_{\alpha
=1}^{r}\left( {\cal T}_{11}^{{\cal H}^{\alpha }}\sum_{j=2}^{\ell }{\cal T}%
_{jj}^{{\cal H}^{\alpha }}-\sum_{j=2}^{\ell }\left( {\cal T}_{1j}^{{\cal H}%
^{\alpha }}\right) ^{2}\right) +{\rm Ric}_{{\cal H}}^{{\cal H}}\left(
U_{1}\right)   \nonumber \\
&&+3\sum\limits_{\alpha =1}^{\ell }\sum\limits_{j=2}^{r}\left( {\cal A}%
_{1j}^{\alpha }\right) ^{2}-\delta \left( N\right) +\left\Vert {\cal T}^{%
{\cal V}}\right\Vert ^{2}-\left\Vert {\cal A}^{{\cal H}}\right\Vert ^{2}.
\label{eq-mx-3H}
\end{eqnarray}%
Using (\ref{eq-RicH}), (\ref{eq-RicV}), (\ref{eq-mx-1H}) in (\ref{eq-mx-3H}%
), we obtain 
\begin{eqnarray}
&&{\rm Ric}_{{\cal V}}^{N_{1}}\left( V_{1}\right) +{\rm Ric}_{{\cal H}%
}^{N_{1}}\left( U_{1}\right) +\sum\limits_{i=1}^{r}{\rm Ric}_{{\cal V}%
}^{N_{1}}\left( U_{i}\right)   \nonumber \\
&=&{\rm Ric}_{{\cal V}}^{{\cal V}}\left( V_{1}\right) +{\rm Ric}_{{\cal H}}^{%
{\cal H}}\left( U_{1}\right) +\sum_{\alpha =1}^{r}\left( {\cal T}_{11}^{%
{\cal H}^{\alpha }}\sum_{j=2}^{\ell }{\cal T}_{jj}^{{\cal H}^{\alpha
}}-\sum_{j=2}^{\ell }\left( {\cal T}_{1j}^{{\cal H}^{\alpha }}\right)
^{2}\right)   \nonumber \\
&&+3\sum\limits_{\alpha =1}^{\ell }\sum\limits_{j=2}^{r}\left( {\cal A}%
_{1j}^{\alpha }\right) ^{2}-\delta \left( N\right) +\left\Vert {\cal T}^{%
{\cal V}}\right\Vert ^{2}-\left\Vert {\cal A}^{{\cal H}}\right\Vert ^{2}.
\label{eq-mx-4H}
\end{eqnarray}%
Using Lemma \ref{Hineva}, we obtain 
\begin{eqnarray}
{\cal T}_{11}^{{\cal H}^{\alpha }}\sum_{j=1}^{\ell }{\cal T}_{jj}^{{\cal H}%
^{\alpha }}-\sum_{j=1}^{\ell }\left( {\cal T}_{1j}^{{\cal H}^{\alpha
}}\right) ^{2} &\geq &\frac{\ell -1}{\ell ^{2}}\left( 2\left(
\sum_{j=1}^{\ell }{\cal T}_{jj}^{{\cal H}^{\alpha }}\right) ^{2}-\ell
\sum_{i,j=1}^{\ell }\left( {\cal T}_{ij}^{{\cal H}^{\alpha }}\right)
^{2}-\left( \ell -2\right) \right.   \nonumber \\
&&\left. \left\vert \sum_{j=1}^{\ell }{\cal T}_{jj}^{{\cal H}^{\alpha
}}\right\vert \sqrt{\frac{\ell \sum_{i,j=1}^{\ell }\left( {\cal T}_{ij}^{%
{\cal H}^{\alpha }}\right) ^{2}-\left( \sum_{j=1}^{\ell }{\cal T}_{jj}^{%
{\cal H}^{\alpha }}\right) ^{2}}{\ell -1}}\right\} .  \label{eq-mx-5H}
\end{eqnarray}%
From (\ref{eq-mx-4H}) and (\ref{eq-mx-5H}), we obtain%
\begin{eqnarray}
&&{\rm Ric}_{{\cal V}}^{N_{1}}\left( V_{1}\right) +{\rm Ric}_{{\cal H}%
}^{N_{1}}\left( U_{1}\right) +\sum\limits_{i=1}^{r}{\rm Ric}_{{\cal V}%
}^{N_{1}}\left( U_{i}\right)   \nonumber \\
&\geq &{\rm Ric}_{{\cal V}}^{{\cal V}}\left( V_{1}\right) +{\rm Ric}_{{\cal H%
}}^{{\cal H}}\left( U_{1}\right) +\frac{\ell -1}{\ell ^{2}}\left(
2\sum_{\alpha =1}^{r}\left( \sum_{j=1}^{\ell }{\cal T}_{jj}^{{\cal H}%
^{\alpha }}\right) ^{2}-\ell \sum_{\alpha =1}^{r}\sum_{i,j=1}^{\ell }\left( 
{\cal T}_{ij}^{{\cal H}^{\alpha }}\right) ^{2}\right.   \nonumber \\
&&\left. -\left( \ell -2\right) \sum_{\alpha =1}^{r}\left\vert
\sum_{j=1}^{\ell }{\cal T}_{jj}^{{\cal H}^{\alpha }}\right\vert \sqrt{\frac{%
\ell \sum_{i,j=1}^{\ell }\left( {\cal T}_{ij}^{{\cal H}^{\alpha }}\right)
^{2}-\left( \sum_{j=1}^{\ell }{\cal T}_{jj}^{{\cal H}^{\alpha }}\right) ^{2}%
}{\ell -1}}\right\}   \nonumber \\
&&+3\sum\limits_{\alpha =1}^{\ell }\sum\limits_{j=2}^{r}\left( {\cal A}%
_{1j}^{\alpha }\right) ^{2}-\delta \left( N\right) +\left\Vert {\cal T}^{%
{\cal V}}\right\Vert ^{2}-\left\Vert {\cal A}^{{\cal H}}\right\Vert ^{2}.
\label{eq-mx-6H}
\end{eqnarray}%
Using Cauchy Squartz inequality, we have%
\begin{equation}
\sum_{\alpha =1}^{r}\left\vert \sum_{j=1}^{\ell }{\cal T}_{jj}^{{\cal H}%
^{\alpha }}\right\vert \sqrt{\frac{\ell \sum_{i,j=1}^{\ell }\left( {\cal T}%
_{ij}^{{\cal H}^{\alpha }}\right) ^{2}-\left( \sum_{j=1}^{\ell }{\cal T}%
_{jj}^{{\cal H}^{\alpha }}\right) ^{2}}{\ell -1}}\leq \ell \left\Vert
H\right\Vert \sqrt{\frac{\ell \left\Vert {\cal T}^{{\cal H}}\right\Vert
^{2}-\ell ^{2}\left\Vert H\right\Vert ^{2}}{\ell -1}}.  \label{eq-mx-7H}
\end{equation}%
From (\ref{eq-mx-6H}) and (\ref{eq-mx-7H}), we obtain%
\begin{eqnarray}
&&{\rm Ric}_{{\cal V}}^{N_{1}}\left( V_{1}\right) +{\rm Ric}_{{\cal H}%
}^{N_{1}}\left( U_{1}\right) +\sum\limits_{i=1}^{r}{\rm Ric}_{{\cal V}%
}^{N_{1}}\left( U_{i}\right)   \nonumber \\
&\geq &{\rm Ric}_{{\cal V}}^{{\cal V}}\left( V_{1}\right) +{\rm Ric}_{{\cal H%
}}^{{\cal H}}\left( U_{1}\right) +\frac{\ell -1}{\ell }\left( 2\ell
\left\Vert H\right\Vert ^{2}-\left\Vert {\cal T}^{{\cal H}}\right\Vert
^{2}\right.   \nonumber \\
&&\left. -\left( \ell -2\right) \sqrt{\frac{\ell \left\Vert H\right\Vert
^{2}\left( \left\Vert {\cal T}^{{\cal H}}\right\Vert ^{2}-\ell \left\Vert
H\right\Vert ^{2}\right) }{\ell -1}}\right)   \nonumber \\
&&+3\sum\limits_{\alpha =1}^{\ell }\sum\limits_{j=2}^{r}\left( {\cal A}%
_{1j}^{\alpha }\right) ^{2}-\delta \left( N\right) +\left\Vert {\cal T}^{%
{\cal V}}\right\Vert ^{2}-\left\Vert {\cal A}^{{\cal H}}\right\Vert ^{2}.
\label{eq-mx-8H}
\end{eqnarray}%
Equality case follows similar computation of Theorem~\ref{Th-Hin-vert-main}.
\end{proof} 

\subsection{Simultaneous Chen--Ricci and Hineva inequalities along the
vertical distribution}

In this subsection, we combine Theorems~\ref{Th-Ric-vert-main} and~\ref%
{Th-Hin-vert-main} to characterize the simultaneous equality cases of the
obtained Ricci curvature inequalities.

\begin{theorem}
\label{Th-ineq-both-side-vert}Let $F:(N_{1},g_{1})\rightarrow (N_{2},g_{2})$
be a Riemannian submersion between two Riemannian manifolds. Then 
\begin{eqnarray}
{\rm Ric}_{{\cal V}}^{N_{1}}(V)-\frac{1}{4}{\ell}^{2}\Vert H\Vert ^{2} &\leq
&{\rm Ric}_{{\cal V}}^{{\cal V}}(V)\leq {\rm Ric}_{{\cal V}}^{N_{1}}(V)-%
\frac{\ell -1}{\ell }\left( 2\ell \Vert H\Vert ^{2}-\Vert {\cal T}^{{\cal H}%
}\Vert ^{2}\right.  \nonumber \\
&&\ \ \ \ \ \ \ \ \ \ \ \ \ \ \ \ \left. \left. -(\ell -2)\sqrt{\frac{\ell
\Vert H\Vert ^{2}\left( \Vert {\cal T}^{{\cal H}}\Vert ^{2}-\ell \Vert
H\Vert ^{2}\right) }{\ell -1}}\right) .\right.  \label{ineq-bot-side}
\end{eqnarray}%
The equalities in both inequalities of {\rm (\ref{ineq-bot-side})} hold
simultaneously for a unit vertical vector $V$ if and only if the matrices 
\[
{\cal T}^{{\cal H}^{\alpha }}={\rm diag}\left( \mu _{\alpha },\ldots ,%
\underset{i^{theorem}\ {\rm position}}{\underbrace{\lambda _{\alpha }}},\ldots
,\mu _{\alpha }\right) ,\quad V=V_{i},\quad i=1,\ldots ,\ell ,\quad \alpha
\in \left\{ 1,\ldots ,r\right\} , 
\]%
and 
\[
\beta _{1}|\lambda _{\alpha }-\mu _{\alpha }|+\beta _{2}\left( \lambda
_{\alpha }+(\ell -1)\mu _{\alpha }\right) =0,\qquad \alpha =1,\ldots ,r, 
\]%
where $\beta _{1}$ and $\beta _{2}$ are scalars but not both zero. Moreover, 
\[
\lambda _{\alpha }=\frac{{\rm trace}({\cal T}^{{\cal H}^{\alpha }})}{2}%
,\qquad \mu _{\alpha }=\frac{{\rm trace}({\cal T}^{{\cal H}^{\alpha }})}{%
2(\ell -1)}. 
\]%
Furthermore, the equalities in both inequalities of {\rm (\ref{ineq-bot-side}%
)} hold simultaneously for every unit vertical vector $V_{i}$, $i=1,\ldots
,\ell $, if and only if either the fibres of $F$ are totally geodesic or $%
\ell =2$ and 
\[
{\cal T}_{11}^{{\cal H}^{\alpha }}={\cal T}_{22}^{{\cal H}^{\alpha }},\qquad 
{\cal T}_{12}^{{\cal H}^{\alpha }}=0,\qquad \alpha =1,\ldots ,r, 
\]%
it is straightforward that the fibers of $F$ are totally umbilical.
\end{theorem}

\begin{proof} Assume that the equalities in both inequalities of (%
\ref{ineq-bot-side}) hold simultaneously for the unit vertical vector $%
V=V_{1}$. From the equality case of Theorem~\ref{Th-Ric-vert-main}, we
obtain 
\begin{equation}
{\cal T}_{1j}^{{\cal H}^{\alpha }}=0,\qquad j=2,\ldots ,\ell ,\quad \alpha
=1,\ldots ,r,  \label{eq-eq-1}
\end{equation}%
and 
\begin{equation}
{\cal T}_{11}^{{\cal H}^{\alpha }}={\cal T}_{22}^{{\cal H}^{\alpha }}+\cdots
+{\cal T}_{\ell \ell }^{{\cal H}^{\alpha }},\qquad \alpha =1,\ldots ,r.
\label{eq-eq-2}
\end{equation}%
On the other hand, the equality case of Theorem~\ref{Th-Hin-vert-main}
yields 
\begin{equation}
{\cal T}_{11}^{{\cal H}^{\alpha }}=\lambda _{\alpha },\qquad {\cal T}_{22}^{%
{\cal H}^{\alpha }}=\cdots ={\cal T}_{\ell \ell }^{{\cal H}^{\alpha }}=\mu
_{\alpha },\qquad \alpha =1,\ldots ,r.  \label{eq-eq-3}
\end{equation}%
Combining (\ref{eq-eq-2}) and (\ref{eq-eq-3}), we obtain%
\[
2T_{11}^{{\cal H}^{\alpha }}={\rm trace}({\cal T}^{{\cal H}^{\alpha }}), 
\]%
which immediately gives 
\[
\lambda _{\alpha }=\frac{{\rm trace}({\cal T}^{{\cal H}^{\alpha }})}{2}. 
\]%
From (\ref{eq-eq-2}) and (\ref{eq-eq-3}), we obtain 
\[
\lambda _{\alpha }=(\ell -1)\mu _{\alpha }, 
\]%
consequently, 
\[
\mu _{\alpha }=\frac{{\rm trace}({\cal T}^{{\cal H}^{\alpha }})}{2(\ell -1)}%
. 
\]%
Suppose that the equalities in (\ref{ineq-bot-side}) hold simultaneously for
every unit vertical vector $V_{i}$, $i=1,\ldots ,\ell $. Then, for each $%
\alpha \in \left\{ 1,\ldots ,s\right\} $, by the equality case of Theorem~%
\ref{Th-Ric-vert-main}, 
\begin{equation}
{\cal T}_{ij}^{{\cal H}^{\alpha }}=0,\qquad i\neq j,\quad i,j\in \{1,\ldots
,\ell \},  \label{eq-sim-mix-1}
\end{equation}%
and 
\begin{equation}
{\cal T}_{ii}^{{\cal H}^{\alpha }}=\sum_{k=1,k\neq i}^{\ell }{\cal T}_{kk}^{%
{\cal H}^{\alpha }},\qquad i=1,\ldots ,\ell .  \label{eq-sim-mix-2}
\end{equation}%
Furthermore, for each $i$, the equality case of Theorem~\ref%
{Th-Hin-vert-main} implies 
\begin{equation}
{\cal T}_{ii}^{{\cal H}^{\alpha }}=\lambda _{\alpha },\qquad {\cal T}_{kk}^{%
{\cal H}^{\alpha }}=\mu _{\alpha },\quad k\neq i,\quad k\in \left\{ 1,\ldots
,i-1,i+1,\ldots ,\ell \right\} .  \label{eq-sim-mix-3}
\end{equation}%
From (\ref{eq-sim-mix-2}) and (\ref{eq-sim-mix-3}), we obtain 
\[
2T_{ii}^{{\cal H}^{\alpha }}=\lambda _{\alpha }+(\ell -1)\mu _{\alpha
},\quad i\in \left\{ 1,\ldots ,\ell \right\} , 
\]%
which leads to 
\[
2\left( \lambda _{\alpha }+(\ell -1)\mu _{\alpha }\right) =\ell \left(
\lambda _{\alpha }+(\ell -1)\mu _{\alpha }\right) . 
\]%
Therefore, 
\[
(\ell -2)\left( \lambda _{\alpha }+(\ell -1)\mu _{\alpha }\right) =0. 
\]%
Thus, 
\[
{\rm either\ }r=2\ {\rm or}\ \lambda _{\alpha }+(\ell -1)\mu _{\alpha }=0. 
\]%
If 
\[
\lambda _{\alpha }+(\ell -1)\mu _{\alpha }=0, 
\]%
then 
\[
{\cal T}_{ii}^{{\cal H}^{\alpha }}=0,\qquad {\cal T}_{ij}^{{\cal H}^{\alpha
}}=0\quad (i\neq j), 
\]%
which implies that ${\cal T}_{ij}^{{\cal H}^{\alpha }}=0$ for all $i,j\in
\left\{ 1,\ldots ,\ell \right\} $ and $\alpha \in \left\{ 1,\ldots
,s\right\} $. Hence, every fibre of $F$ is totally geodesic. If $\ell =2$,
then 
\[
{\cal T}_{11}^{{\cal H}^{\alpha }}={\cal T}_{22}^{{\cal H}^{\alpha }},\qquad 
{\cal T}_{12}^{{\cal H}^{\alpha }}=0, 
\]%
it is straightforward that the fibers of $F$ are totally umbilical.
\end{proof} 

\subsection{Simultaneous Chen--Ricci and Hineva inequalities along the mixed
distribution}

In this subsection, we combine Theorems~\ref{Th-Ric-mix-main} and~\ref%
{Th-Hin-mix-main} to characterize the simultaneous equality cases of the
obtained Ricci curvature inequalities.

\begin{theorem}
\label{Th-sim-chen-Hin-mix} Let $F:(N_{1},g_{1})\rightarrow (N_{2},g_{2})$
be a Riemannian submersion between two Riemannian manifolds. Then%
\begin{align}
& {\rm Ric}_{{\cal V}}^{N_{1}}\left( V_{1}\right) +{\rm Ric}_{{\cal H}%
}^{N_{1}}\left( U_{1}\right) +\sum\limits_{i=1}^{r}{\rm Ric}_{{\cal V}%
}^{N_{1}}\left( U_{i}\right)  \nonumber \\
& -\frac{\ell ^{2}}{4}\left\Vert H\right\Vert ^{2}-3\sum\limits_{\alpha
=1}^{\ell }\sum\limits_{j=2}^{r}\left( {\cal A}_{1j}^{\alpha }\right)
^{2}+\delta \left( N\right) -\left\Vert {\cal T}^{{\cal V}}\right\Vert
^{2}+\left\Vert {\cal A}^{{\cal H}}\right\Vert ^{2}  \nonumber \\
& \leq {\rm Ric}_{{\cal V}}^{{\cal V}}\left( V_{1}\right) +{\rm Ric}_{{\cal H%
}}^{{\cal H}}\left( U_{1}\right) \leq {\rm Ric}_{{\cal V}}^{N_{1}}\left(
V_{1}\right) +{\rm Ric}_{{\cal H}}^{N_{1}}\left( U_{1}\right)
+\sum\limits_{i=1}^{r}{\rm Ric}_{{\cal V}}^{N_{1}}\left( U_{i}\right) 
\nonumber \\
& -\frac{\ell -1}{\ell }\left( 2\ell \left\Vert H\right\Vert ^{2}-\left\Vert 
{\cal T}^{{\cal H}}\right\Vert ^{2}-\left( \ell -2\right) \sqrt{\frac{\ell
\left\Vert H\right\Vert ^{2}\left( \left\Vert {\cal T}^{{\cal H}}\right\Vert
^{2}-\ell \left\Vert H\right\Vert ^{2}\right) }{\ell -1}}\right)  \nonumber
\\
& -3\sum\limits_{\alpha =1}^{\ell }\sum\limits_{j=2}^{r}\left( {\cal A}%
_{1j}^{\alpha }\right) ^{2}+\delta \left( N\right) -\left\Vert {\cal T}^{%
{\cal V}}\right\Vert ^{2}+\left\Vert {\cal A}^{{\cal H}}\right\Vert ^{2}.
\label{ineq-both-mix}
\end{align}%
The equality case follows Theorem~\ref{Th-ineq-both-side-vert}.
\end{theorem}

\subsection{Applications to real and complex space forms}

In this subsection, we apply the general results established in Theorems~\ref%
{Th-ineq-both-side-vert} and \ref{Th-sim-chen-Hin-mix} to Riemannian
submersions whose total spaces are real and complex space forms. In
particular, we obtain the Chen--Ricci and Hineva inequalities
simultaneously, together with a complete characterization of their
simultaneous equality cases.

\begin{proposition}
Let $F:\left( N_{1},g_{1}\right) \rightarrow \left( N_{2},g_{2}\right) $ be
a Riemannian submersion from real space form onto a Riemannian manifold with 
$\dim N_{1}=n_{1}$ and $\dim N_{2}=n_{2}$. Then

\begin{enumerate}
\item[$(1)$] By using {\rm (\ref{eq-RicV})} and {\rm (\ref{eq-RSF})}, we
obtain%
\begin{equation}
{\rm Ric}_{{\cal V}}^{N_{1}}(v_{1})=c\left( \ell -1\right) .
\label{eq-RSF-Ric-1}
\end{equation}

\item[$(2)$] By using {\rm (\ref{eq-RicH})} and {\rm (\ref{eq-RSF})}, we
obtain 
\begin{equation}
{\rm Ric}_{{\cal H}}^{N_{1}}(U_{1})=c\left( r-1\right) .
\label{eq-RSF-Ric-2}
\end{equation}

\item[$(3)$] By using {\rm (\ref{eq-RSF})} and {\rm (\ref{eq-Ric-UV})}, we
obtain%
\begin{equation}
\sum\limits_{i=1}^{r}{\rm Ric}_{{\cal V}}^{N_{1}}\left( U_{i}\right) =c\ell
r.  \label{eq-RSF-Ric-3}
\end{equation}
\end{enumerate}
\end{proposition}

\begin{proposition}
Let $F:\left( N_{1},g_{1}\right) \rightarrow \left( N_{2},g_{2}\right) $ be
a Riemannian submersion from complex space form onto a Riemannian manifold
with $\dim N_{1}=n_{1}$ and $\dim N_{2}=n_{2}$. Then

\begin{enumerate}
\item[$(1)$] By using {\rm (\ref{eq-RicV})} and {\rm (\ref{eq-GCSF})}, we
obtain 
\begin{equation}
{\rm Ric}_{{\cal V}}^{N_{1}}(v_{1})=\frac{c}{4}\left( \ell -1\right) +\frac{%
3c}{4}\left\Vert Qv_{1}\right\Vert ^{2}.  \label{eq-CSF-Ric-1}
\end{equation}

\item[$(2)$] By using {\rm (\ref{eq-RicH})} and {\rm (\ref{eq-GCSF})}, we
obtain 
\begin{equation}
{\rm Ric}_{{\cal H}}^{N_{1}}(U_{1})=\frac{c}{4}\left( r-1\right) +\frac{3c}{4%
}\left\Vert Ph_{1}\right\Vert ^{2}.  \label{eq-CSF-Ric-2}
\end{equation}

\item[$(3)$] By using {\rm (\ref{eq-GCSF})} and {\rm (\ref{eq-Ric-UV})}, we
obtain 
\begin{equation}
\sum\limits_{i=1}^{r}{\rm Ric}_{{\cal V}}^{N_{1}}\left( U_{i}\right) =\frac{c%
}{4}\ell r+\frac{3c}{4}\left\Vert P^{{\cal V}}\right\Vert ^{2}.
\label{eq-CSF-Ric-3}
\end{equation}
\end{enumerate}
\end{proposition}

\begin{theorem}
Let $F:\left( N_{1},g_{1}\right) \rightarrow \left( N_{2},g_{2}\right) $ be
a Riemannian submersion from real space form onto a Riemannian manifold with 
$\dim N_{1}=n_{1}$ and $\dim N_{2}=n_{2}$. Then 
\begin{eqnarray}
c\left( \ell -1\right) -\frac{1}{4}{\ell}^{2}\left\Vert H\right\Vert ^{2}
&\leq &{\rm Ric}_{{\cal V}}^{{\cal V}}\left( V_{1}\right) \leq c\left( \ell
-1\right) -\frac{\ell -1}{\ell }\left( 2\ell \left\Vert H\right\Vert
^{2}-\left\Vert {\cal T}^{{\cal H}}\right\Vert ^{2}\right.  \nonumber \\
&&\ \ \ \ \ \ \ \ \ \ \left. \left. -\left( \ell -2\right) \sqrt{\frac{\ell
\left\Vert H\right\Vert ^{2}\left( \left\Vert {\cal T}^{{\cal H}}\right\Vert
^{2}-\ell \left\Vert H\right\Vert ^{2}\right) }{\ell -1}}\right) \right. .
\label{ineq-both-RSF-vert}
\end{eqnarray}%
\begin{eqnarray}
&&c\left( \ell r+\ell +r-2\right) -\frac{\ell ^{2}}{4}\left\Vert
H\right\Vert ^{2}-3\sum\limits_{\alpha =1}^{\ell
}\sum\limits_{j=2}^{r}\left( {\cal A}_{1j}^{\alpha }\right) ^{2}+\delta
\left( N\right) -\left\Vert {\cal T}^{{\cal V}}\right\Vert ^{2}+\left\Vert 
{\cal A}^{{\cal H}}\right\Vert ^{2}  \nonumber \\
&\leq &{\rm Ric}_{{\cal V}}^{{\cal V}}\left( V_{1}\right) +{\rm Ric}_{{\cal H%
}}^{{\cal H}}\left( U_{1}\right) \leq c\left( \ell r+\ell +r-2\right) -\frac{%
\ell -1}{\ell }\left( 2\ell \left\Vert H\right\Vert ^{2}-\left\Vert {\cal T}%
^{{\cal H}}\right\Vert ^{2}\right.  \nonumber \\
&&\left. -\left( \ell -2\right) \sqrt{\frac{\ell \left\Vert H\right\Vert
^{2}\left( \left\Vert {\cal T}^{{\cal H}}\right\Vert ^{2}-\ell \left\Vert
H\right\Vert ^{2}\right) }{\ell -1}}\right) -3\sum\limits_{\alpha =1}^{\ell
}\sum\limits_{j=2}^{r}\left( {\cal A}_{1j}^{\alpha }\right) ^{2}+\delta
\left( N\right)  \nonumber \\
&&-\left\Vert {\cal T}^{{\cal V}}\right\Vert ^{2}+\left\Vert {\cal A}^{{\cal %
H}}\right\Vert ^{2}.  \label{ineq-both-RSF-mixed}
\end{eqnarray}%
The equality case follows {\rm Theorem~\ref{Th-ineq-both-side-vert}}.
\end{theorem}

\begin{proof} By using (\ref{eq-RSF-Ric-1}) in (\ref{ineq-bot-side}%
), we obtain (\ref{ineq-both-RSF-vert}). Further, by using (\ref%
{eq-RSF-Ric-1}), (\ref{eq-RSF-Ric-2}) and (\ref{eq-RSF-Ric-3}) in (\ref%
{ineq-both-mix}), we obtain (\ref{ineq-both-RSF-mixed}).
\end{proof} 

\begin{theorem}
Let $F:\left( N_{1},g_{1}\right) \rightarrow \left( N_{2},g_{2}\right) $ be
a Riemannian submersion from complex space form onto a Riemannian manifold
with $\dim N_{1}=n_{1}$ and $\dim N_{2}=n_{2}$. Then 
\begin{eqnarray}
&&\left. \frac{c}{4}\left( \ell -1\right) +\frac{3c}{4}\left\Vert
Qv_{1}\right\Vert ^{2}-\frac{1}{4}{\ell}^{2}\left\Vert H\right\Vert ^{2}\leq 
{\rm Ric}_{{\cal V}}^{{\cal V}}\left( V_{1}\right) \leq \frac{c}{4}\left(
\ell -1\right) +\frac{3c}{4}\left\Vert Qv_{1}\right\Vert ^{2}\right. 
\nonumber \\
&&-\frac{\ell -1}{\ell }\left( 2\ell \left\Vert H\right\Vert ^{2}-\left\Vert 
{\cal T}^{{\cal H}}\right\Vert ^{2}-\left( \ell -2\right) \sqrt{\frac{\ell
\left\Vert H\right\Vert ^{2}\left( \left\Vert {\cal T}^{{\cal H}}\right\Vert
^{2}-\ell \left\Vert H\right\Vert ^{2}\right) }{\ell -1}}\right)
\label{ineq-both-CSF-vert}
\end{eqnarray}%
\begin{eqnarray}
&&\frac{c}{4}\left( \ell r+\ell +r-2\right) +\frac{3c}{4}\left( \left\Vert
Qv_{1}\right\Vert ^{2}+\left\Vert Ph_{1}\right\Vert ^{2}+\left\Vert P^{{\cal %
V}}\right\Vert ^{2}\right) -\frac{\ell ^{2}}{4}\left\Vert H\right\Vert
^{2}-3\sum\limits_{\alpha =1}^{\ell }\sum\limits_{j=2}^{r}\left( {\cal A}%
_{1j}^{\alpha }\right) ^{2}  \nonumber \\
&&\left. +\delta \left( N\right) -\left\Vert {\cal T}^{{\cal V}}\right\Vert
^{2}+\left\Vert {\cal A}^{{\cal H}}\right\Vert ^{2}\leq {\rm Ric}_{{\cal V}%
}^{{\cal V}}\left( V_{1}\right) +{\rm Ric}_{{\cal H}}^{{\cal H}}\left(
U_{1}\right) \leq \frac{c}{4}\left( \ell r+\ell +r-2\right) \right. 
\nonumber \\
&&+\frac{3c}{4}\left( \left\Vert Qv_{1}\right\Vert ^{2}+\left\Vert
Ph_{1}\right\Vert ^{2}+\left\Vert P^{{\cal V}}\right\Vert ^{2}\right) -\frac{%
\ell -1}{\ell }\left( 2\ell \left\Vert H\right\Vert ^{2}-\left\Vert {\cal T}%
^{{\cal H}}\right\Vert ^{2}\right.  \nonumber \\
&&\left. -\left( \ell -2\right) \sqrt{\frac{\ell \left\Vert H\right\Vert
^{2}\left( \left\Vert {\cal T}^{{\cal H}}\right\Vert ^{2}-\ell \left\Vert
H\right\Vert ^{2}\right) }{\ell -1}}\right) -3\sum\limits_{\alpha =1}^{\ell
}\sum\limits_{j=2}^{r}\left( {\cal A}_{1j}^{\alpha }\right) ^{2}+\delta
\left( N\right)  \nonumber \\
&&-\left\Vert {\cal T}^{{\cal V}}\right\Vert ^{2}+\left\Vert {\cal A}^{{\cal %
H}}\right\Vert ^{2}.  \label{ineq-both-CSF-mixed}
\end{eqnarray}%
The equality case follows {\rm Theorem~\ref{Th-ineq-both-side-vert}}.
\end{theorem}

\begin{proof}By using (\ref{eq-CSF-Ric-1}) in (\ref{ineq-bot-side}%
), we obtain (\ref{ineq-both-CSF-vert}). Further, by using (\ref%
{eq-CSF-Ric-1}), (\ref{eq-CSF-Ric-2}) and (\ref{eq-CSF-Ric-3}) in (\ref%
{ineq-both-mix}), we obtain (\ref{ineq-both-CSF-mixed}).
\end{proof}

\section{Background of Riemannian maps}

\label{sec_2} %%%%%%%%%%%%%%%%%%%%%%%%%%%%%%%%%%%%%
In this section, we collect some basic information that is needed for the
remaining subsections.\newline
Let $F:(N_{1},g_{1})\rightarrow (N_{2},g_{2})$ be a smooth map between
Riemannian manifolds of dimensions $n_{1}$ and $n_{2}$, respectively, with
rank $r<\min \{n_{1},n_{2}\}$, and let $F_{\ast p}:T_{p}N_{1}\rightarrow
T_{F(p)}N_{2}$ be its derivative map at $p$. Denoting the kernel (resp.
range) space of $F_{\ast }$ by ${\cal V}_{p}=(\ker F_{\ast p})$ (resp. $%
{\cal R}_{p}=({\rm range~}F_{\ast p})$) and its orthogonal complementary
space by ${\cal H}_{p}=(\ker F_{\ast p})^{\perp }$ (resp. ${\cal R}%
_{p}^{\perp }=({\rm range~}F_{\ast p})^{\perp }$), we decompose 
\[
T_{p}N_{1}={\cal V}_{p}\oplus {\cal H}_{p}~{\rm and~}T_{F(p)}N_{2}={\cal R}%
_{p}\oplus {\cal R}_{p}^{\perp }. 
\]%
The map $F$ is called a {\it Riemannian map}, if for all $X,Y\in {\cal H}$,
we have \cite{Fischer_1992} 
\begin{equation}
g_{1}(X,Y)=g_{2}(F_{\ast }X,F_{\ast }Y).  \label{eq-(1)}
\end{equation}%
Let $\left\{ U_{1},\dots ,U_{r}\right\} $, $\{F_{\ast }U_{1},\dots ,F_{\ast
}U_{r}\}$ and $\{V_{r+1},\dots ,V_{n_{2}}\}$ be orthonormal bases for ${\cal %
H}_{p}$, ${\cal R}_{p}$, and ${\cal R}_{p}^{\perp }$, respectively. These
orthonormal bases will be used throughout the subsequent subsections.

\subsubsection*{Second fundamental form}

\noindent The bundle {\rm Hom}$\left( T{N_{1}},F^{-1}T{N_{2}}\right) $
admits an induced connection $\nabla $ from the Levi-Civita connection $%
\nabla ^{1}$ on {$N$}${_{1}}$. The symmetric {\it second fundamental form}
of $F$ is then given by \cite{Nore_1986} 
\[
\left( \nabla F_{\ast }\right) \left( Z_{1},Z_{2}\right) =\nabla
_{Z_{1}}^{F}F_{\ast }(Z_{2})-F_{\ast }\left( \nabla _{Z_{1}}^{{1}%
}Z_{2}\right) , 
\]%
where $Z_{1},Z_{2}\in T${$N$}${_{1}}$ and $\nabla ^{F}$ is the pull-back
connection. In addition, by \cite{Sahin_2010} $\left( \nabla F_{\ast
}\right) \left( X,Y\right) $ is completely contained in ${\cal R}^{\perp }~%
{\rm for\ all}\ X,Y\in {\cal H}$.

\subsubsection*{Gauss equation}

The {\it Gauss equation} for $F$ is defined as \cite[p. 189]{Sahin_book} 
\begin{eqnarray}
g_{2}\left( R^{{N_{2}}}\left( F_{\ast }Z_{1},F_{\ast }Z_{2}\right) F_{\ast
}Z_{3},F_{\ast }Z_{4}\right) &=&g_{1}\left( R^{{N_{1}}}\left(
Z_{1},Z_{2}\right) Z_{3},Z_{4}\right)  \nonumber \\
&&+g_{2}\left( \left( \nabla F_{\ast }\right) \left( Z_{1},Z_{3}\right)
,\left( \nabla F_{\ast }\right) \left( Z_{2},Z_{4}\right) \right)  \nonumber
\\
&&-g_{2}\left( \left( \nabla F_{\ast }\right) \left( Z_{1},Z_{4}\right)
,\left( \nabla F_{\ast }\right) \left( Z_{2},Z_{3}\right) \right) ,
\label{Gauss_Eqn}
\end{eqnarray}%
where $Z_{i}\in {\cal H}$. Here, $R^{N_{1}}$ and $R^{N_{2}}$ denote the
curvature tensors of {$N$}${_{1}}$ and {$N$}${_{2}}$, respectively.

\subsubsection*{Some notations}

The Ricci curvatures on the horizontal and range spaces are given,
respectively, by 
\begin{equation}
{\rm Ric}_{{\cal H}}^{N_{1}}(U_{i})=\sum_{j=1}^{r}g_{1}\left( R^{{N_{1}}%
}(U_{i},U_{j})U_{j},U_{i}\right) ,\ {\rm Ric}_{{\cal R}}^{N_{2}}(F_{\ast
}U_{i})=\sum_{j=1}^{r}g_{2}\left( R^{{N_{2}}}(F_{\ast }U_{i},F_{\ast
}U_{j})F_{\ast }U_{j},F_{\ast }U_{i}\right) .  \label{Ricci_Curvatures}
\end{equation}%
We set, 
\begin{eqnarray*}
B_{ij}^{{\cal H}^{\alpha }} &=&g_{2}\left( (\nabla F_{\ast
})(U_{i},U_{j}),V_{\alpha }\right) ,\quad 1\leq i,j\leq r,\quad r+1\leq
\alpha \leq n_{2}, \\
\left\Vert B^{{\cal H}}\right\Vert ^{2} &=&\sum_{i,j=1}^{r}g_{2}\left(
(\nabla F_{\ast })(U_{i},U_{j}),(\nabla F_{\ast })(U_{i},U_{j})\right) , \\
{\rm trace\,}\left( B^{{\cal H}^{\alpha }}\right) &=&\sum_{i=1}^{r}B_{ii}^{%
{\cal H}^{\alpha }},\ {\rm trace\,}\left( B^{{\cal H}^{\alpha }}\right)
^{2}=\sum_{i,j=1}^{r}\left( B_{ij}^{{\cal H}^{\alpha }}\right) ^{2},\quad
\alpha \in \left\{ r+1,\ldots ,n_{2}\right\} ,
\end{eqnarray*}%
\begin{equation}
{\rm trace\,}B^{{\cal H}}=\sum_{i=1}^{r}(\nabla F_{\ast })\left(
U_{i},U_{i}\right) ,\quad \left\Vert {\rm trace\,}B^{{\cal H}}\right\Vert
^{2}=g_{2}\left( {\rm trace\,}B^{{\cal H}},{\rm trace\,}B^{{\cal H}}\right) .
\label{eq-trace-BH}
\end{equation}%
We begin with the following examples of Riemannian maps, which will be used
to illustrate the sharpness and applicability of the results established in
the subsequent sections.

\begin{example}
Consider the Riemannian map described in {\rm \cite[Example~29]%
{Singh_Meena_Meena_26_JMAA}}. For this map, the equality case of Theorem~\ref%
{Th-Ric-RM} is not attained. On the other hand, the equality case of Theorem~%
\ref{Th-RM-Hin} is attained only for the unit vector $X=U_{1}$ and fails to
hold for the remaining unit vectors. This follows from the fact that the
corresponding shape operator is given by%
\[
S_{4}=\left( 
\begin{array}{ccc}
-\dfrac{1}{y_{1}^{2}-1} & 0 & 0 \\ 
0 & 0 & 0 \\ 
0 & 0 & 0%
\end{array}%
\right) . 
\]
\end{example}

\begin{example}
Consider the Riemannian map described in {\rm \cite[Example~30]%
{Singh_Meena_Meena_26_JMAA}}. For this map, the equality cases of Theorems~%
\ref{Th-Ric-RM} and~\ref{Th-RM-Hin} are attained for every unit vector $%
X=U_{i}$, $i\in\{1,2,3\}$. Indeed, the corresponding shape operator is given
by 
\[
S_{4}= \left( 
\begin{array}{ccc}
0 & 0 & 0 \\ 
0 & 0 & 0 \\ 
0 & 0 & 0%
\end{array}
\right). 
\]
\end{example}

\begin{example}
Consider the Riemannian map described in {\rm \cite[Example~4.2]{SMM_CF_2025}%
}. For this map, the equality cases of Theorems~\ref{Th-Ric-RM} is not
attained. On the other hand, the equality case of Theorem~\ref{Th-RM-Hin} is
attained only for the unit vector $X=U_{2}$ and fails to hold for the
remaining unit vectors. This follows from the fact that the corresponding
shape operator are given by 
\[
S_{5}=\left( 
\begin{array}{cccc}
0 & 0 & 0 & 0 \\ 
0 & \dfrac{2}{4-y_{2}^{2}} & 0 & 0 \\ 
0 & 0 & 0 & 0 \\ 
0 & 0 & 0 & 0%
\end{array}%
\right) ,\qquad S_{6}=\left( 
\begin{array}{cccc}
0 & 0 & 0 & 0 \\ 
0 & \dfrac{2}{4-y_{4}^{2}} & 0 & 0 \\ 
0 & 0 & 0 & 0 \\ 
0 & 0 & 0 & 0%
\end{array}%
\right) . 
\]
\end{example}

\subsection{Chen-Ricci inequality for Riemannian maps}

The following theorem was proved in {\rm \cite[Theorem~3.1, Eq.~(3.11)]%
{SMM_CR_2025}} using scalar curvature. In this paper, we present an
alternative proof by employing Ricci curvature.

\begin{theorem}
{\rm \cite[Theorem~4.1]{SMM_CR_2025}} \label{Th-Ric-RM}Let $F:\left(
N_{1},g_{1}\right) \rightarrow \left( N_{2},g_{2}\right) $ be a Riemannian
map between Riemannian manifolds with rank $r<n_{2}$. Then for any vector $%
U\in {\cal H}_{p}$,%
\[
{\rm Ric}^{{\cal H}}\left( U\right) \leq {\rm Ric}_{{\cal R}}^{N_{2}}\left(
F_{\ast }U\right) +\frac{1}{4}\left\Vert {\rm trace}B^{{\cal H}}\right\Vert
^{2}. 
\]%
The equality holds if and only if, for each $\alpha \in \left\{ r+1,\ldots
,n_{2}\right\} $ and $U=U_{i}$,%
\[
B_{ij}^{{\cal H}^{\alpha }}=0,\quad j\in \left\{ 1,\ldots i-1,i+1,\ldots
,r\right\} 
\]%
\[
B_{ii}^{{\cal H}^{\alpha }}=B_{11}^{{\cal H}^{\alpha }}+\cdots +B_{i-1i-1}^{%
{\cal H}^{\alpha }}+B_{i+1i+1}^{{\cal H}^{\alpha }}+\cdots +B_{rr}^{{\cal H}%
^{\alpha }}. 
\]
Equivalently, the equality holds for all unit vector fields if and only if
the map $F$ is

\begin{enumerate}
\item totally geodesic when $r>2;$

\item totally umbilical when $r=2.$
\end{enumerate}
\end{theorem}

\begin{proof}
Suppose $U=U_{1}$, Then by using (\ref{Gauss_Eqn})
and (\ref{Ricci_Curvatures}), we obtain 
\begin{equation}
{\rm Ric}_{{\cal H}}^{N_{1}}\left( U_{1}\right) ={\rm Ric}_{{\cal R}%
}^{N_{2}}\left( F_{\ast }U_{1}\right) +\sum_{\alpha =r+1}^{n_{2}}\left(
B_{11}^{{\cal H}^{\alpha }}\sum_{j=2}^{r}B_{jj}^{{\cal H}^{\alpha
}}-\sum_{j=2}^{r}\left( B_{1j}^{{\cal H}^{\alpha }}\right) ^{2}\right) .
\label{eq-RM-1}
\end{equation}%
Thus, we obtain%
\begin{equation}
{\rm Ric}_{{\cal H}}^{N_{1}}\left( U_{1}\right) \leq {\rm Ric}_{{\cal R}%
}^{N_{2}}\left( F_{\ast }U_{1}\right) +\sum_{\alpha =r+1}^{n_{2}}\left(
B_{11}^{{\cal H}^{\alpha }}\sum_{j=2}^{r}B_{jj}^{{\cal H}^{\alpha }}\right) .
\label{eq-RM-2}
\end{equation}%
Now, we have the basic inequality%
\begin{equation}
4d_{1}d_{2}\leq \left( d_{1}+d_{2}\right) ^{2},  \label{eq-RM-3}
\end{equation}%
with equality if and only if $d_{1}=d_{2}$. Assume $d_{1}=B_{11}^{{\cal H}%
^{\alpha }}$ and $d_{2}=\sum_{j=2}^{r}B_{jj}^{{\cal H}^{\alpha }}$. Then, by
using (\ref{eq-RM-2}) and (\ref{eq-RM-3}), we obtain%
\begin{equation}
{\rm Ric}_{{\cal H}}^{N_{1}}\left( U_{1}\right) \leq {\rm Ric}_{{\cal R}%
}^{N_{2}}\left( F_{\ast }U_{1}\right) +\frac{1}{4}\sum_{\alpha
=r+1}^{n_{2}}\left( B_{11}^{{\cal H}^{\alpha }}+\sum_{j=2}^{r}B_{jj}^{{\cal H%
}^{\alpha }}\right) ^{2}.  \label{eq-RM-4}
\end{equation}%
From (\ref{eq-trace-BH}) and (\ref{eq-RM-4}), we obtain%
\begin{equation}
{\rm Ric}_{{\cal H}}^{N_{1}}\left( U_{1}\right) \leq {\rm Ric}_{{\cal R}%
}^{N_{2}}\left( F_{\ast }U_{1}\right) +\frac{1}{4}\left\Vert {\rm trace}B^{%
{\cal H}}\right\Vert ^{2}.  \label{eq-RM-5}
\end{equation}%
The equality holds in (\ref{eq-RM-5}) if and only if equality holds in (\ref%
{eq-RM-4}), for each $\alpha \in \left\{ r+1,\ldots ,n_{2}\right\} $, we
obtain%
\begin{equation}
B_{11}^{{\cal H}^{\alpha }}=B_{22}^{{\cal H}^{\alpha }}+\cdots +B_{\ell \ell
}^{{\cal H}^{\alpha }},\quad \quad B_{1j}^{{\cal H}^{\alpha }}=0,\quad j\in
\{2,\ldots ,r\}.  \label{eq-RM-6}
\end{equation}%
Assuming that equality holds for all unit vector field $U_{i}$, $i\in
\{1,\ldots ,r\}$. Then from (\ref{eq-RM-6}), for each $\alpha \in \left\{
r+1,\ldots ,n_{2}\right\} ,$ we obtain%
\begin{equation}
B_{ij}^{{\cal H}^{\alpha }}=0,\quad i\neq j,  \label{eq-RM-7}
\end{equation}%
\begin{equation}
2B_{ii}^{{\cal H}^{\alpha }}=B_{11}^{{\cal H}^{\alpha }}+B_{22}^{{\cal H}%
^{\alpha }}+\cdots +B_{\ell \ell }^{{\cal H}^{\alpha }},\quad i\in
\{1,\ldots r\}.  \label{eq-RM-8}
\end{equation}%
From (\ref{eq-RM-8}), we obtain%
\begin{equation}
\left( r-2\right) (B_{11}^{{\cal H}^{\alpha }}+B_{22}^{{\cal H}^{\alpha
}}+\cdots +B_{\ell \ell }^{{\cal H}^{\alpha }})=0.  \label{eq-RM-9}
\end{equation}%
From (\ref{eq-RM-9}), we obtain either $r=2$ or $B_{11}^{{\cal H}^{\alpha
}}+B_{22}^{{\cal H}^{\alpha }}+\cdots +B_{rr}^{{\cal H}^{\alpha }}=0$. If $%
B_{11}^{{\cal H}^{\alpha }}+B_{22}^{{\cal H}^{\alpha }}+\cdots +B_{rr}^{%
{\cal H}^{\alpha }}=0$, then from (\ref{eq-RM-8}), we obtain $B_{ii}^{{\cal H%
}^{\alpha }}=0$ for all $\alpha \in \left\{ r+1,\ldots ,n_{2}\right\} $ and $%
i\in \{1,\ldots ,r\}$. In view of (\ref{eq-RM-7}), we obtain $B_{ij}^{{\cal H%
}^{\alpha }}=0$, for all $\alpha \in \left\{ r+1,\ldots ,n_{2}\right\} $ and 
$i,j\in \{1,\ldots ,r\}$. Which shows $F$ is totally geodesic. If $r=2$,
then from (\ref{eq-RM-7}) and (\ref{eq-RM-8}), for each $\alpha \in \left\{
r+1,\ldots ,n_{2}\right\} $, we obtain 
\begin{equation}
B_{11}^{{\cal H}^{\alpha }}=B_{22}^{{\cal H}^{\alpha }},\quad B_{12}^{{\cal H%
}^{\alpha }}=0.  \label{eq-RM-10}
\end{equation}%
From (\ref{eq-RM-10}), it is straightforward that the $F$ is totally
umbilical.
\end{proof}

\subsection{Hineva inequality for Riemannian maps}

We begin with the following theorem, which establishes the Hineva inequality
providing an lower bound for the Ricci curvature of a Riemannian maps.

\begin{theorem}
\label{Th-RM-Hin} Let $F:\left( N_{1}^{n_{1}},g_{1}\right) \rightarrow
\left( N_{2}^{n_{2}},g_{2}\right) $ be a Riemannian map between Riemannian
manifolds with rank $r<n_{2}$. Then for any vector $U\in {\cal H}_{p}$, 
\begin{eqnarray}
{\rm Ric}_{{\cal H}}^{N_{1}}\left( U\right) &\geq &{\rm Ric}_{{\cal R}%
}^{N_{2}}\left( F_{\ast }U\right) +\frac{r-1}{r^{2}}\left( 2\left\Vert {\rm %
trace}B^{{\cal H}}\right\Vert ^{2}-r\left\Vert B^{{\cal H}}\right\Vert
^{2}\right.  \nonumber \\
&&\left. -\left( r-2\right) \left\Vert {\rm trace}B^{{\cal H}}\right\Vert 
\sqrt{\frac{r\left\Vert B^{{\cal H}}\right\Vert ^{2}-\left\Vert {\rm trace}%
B^{{\cal H}}\right\Vert ^{2}}{r-1}}\right) .  \label{eq-GINQ-RM}
\end{eqnarray}%
The equality holds in {\rm (\ref{eq-GINQ-RM})} if and only if the matrices $%
\left( B_{ij}^{{\cal H}^{\alpha }}\right) $ take the following forms:%
\[
B^{{\cal H}^{\alpha }}={\rm diag}\left( b_{\alpha },\ldots ,\underset{%
i^{theorem}\ {\rm position}}{\underbrace{a_{\alpha }}},\ldots ,b_{\alpha }\right)
,\quad U=U_{i},\quad i=1,\ldots ,r,\quad \alpha \in \left\{ r+1,\ldots
,n_{2}\right\} , 
\]%
and%
\[
\beta _{1}|a_{\alpha }-b_{\alpha }|+\beta _{2}\left( a_{\alpha
}+(r-1)b_{\alpha }\right) =0,\qquad \alpha =r+1,\ldots ,n_{2}, 
\]%
where $\beta _{1}$ and $\beta _{2}$ are constants. where $a_{\alpha }$ and $%
b_{\alpha }$ are two eigen value of each one of the matrices $\left( B_{ij}^{%
{\cal H}^{\alpha }}\right) $ of the symmetric $\left( 1,2\right) $ tensor
field $\nabla F_{\ast }$ such that 
\begin{eqnarray*}
a_{\alpha } &=&\frac{{\rm trace}\left( B^{{\cal H}^{\alpha }}\right) }{r}\mp 
\frac{r-1}{r}\sqrt{\frac{r\,{\rm trace}\left( B^{{\cal H}^{\alpha }}\right)
^{2}-\left( {\rm trace}\left( B^{{\cal H}^{\alpha }}\right) \right) ^{2}}{r-1%
}}, \\
b_{\alpha } &=&\frac{{\rm trace}\left( B^{{\cal H}^{\alpha }}\right) }{r}\pm 
\frac{1}{r}\sqrt{\frac{r\,{\rm trace}\left( B^{{\cal H}^{\alpha }}\right)
^{2}-\left( {\rm trace}\left( B^{{\cal H}^{\alpha }}\right) \right) ^{2}}{r-1%
}}.
\end{eqnarray*}%
The equality conditions can be interpreted as follows. We observe that $%
g_{2}\left( F_{\ast }U_{i},S_{V_{\alpha }}F_{\ast }U_{i}\right) \neq
g_{2}\left( F_{\ast }U_{1},S_{V_{\alpha }}F_{\ast }U_{1}\right) =\cdots
=g_{2}\left( F_{\ast }U_{i-1},S_{V_{\alpha }}F_{\ast }U_{i-1}\right)
=g_{2}\left( F_{\ast }U_{i+1},S_{V_{\alpha }}F_{\ast }U_{i+1}\right) =\cdots
=g_{2}\left( F_{\ast }U_{r},S_{V_{\alpha }}F_{\ast }U_{r}\right) $ with
respect to all directions $(V_{\alpha },\text{where}~\alpha \in \{r+1,\dots
,n_{2}\})$. Equivalently, there exist $(n_{2}-r)$ mutually orthogonal unit
vector fields in $\left( {\rm range~}F_{\ast }\right) ^{\perp }$ such that
shape operators with respect to all directions have an eigenvalue of
multiplicity $(r-1)$ and that for each $V_{\alpha }$ the distinguished
eigendirections are the same (namely $F_{\ast }Z_{i}$). Hence, the leaves of
range spaces are quasi-umbilical $\cite{DHV_2008}$.
\end{theorem}

\begin{proof} Suppose $U=U_{1}$, Then by using (\ref{Gauss_Eqn}) and
(\ref{Ricci_Curvatures}), we obtain 
\begin{equation}
{\rm Ric}_{{\cal H}}^{N_{1}}\left( U_{1}\right) ={\rm Ric}_{{\cal R}%
}^{N_{2}}\left( F_{\ast }U_{1}\right) +\sum_{\alpha =r+1}^{n_{2}}\left(
B_{11}^{{\cal H}^{\alpha }}\sum_{j=2}^{r}B_{jj}^{{\cal H}^{\alpha
}}-\sum_{j=2}^{r}\left( B_{1j}^{{\cal H}^{\alpha }}\right) ^{2}\right) .
\label{eq-Ric-RM-1}
\end{equation}%
Using Lemma \ref{Hineva}, we obtain%
\begin{eqnarray}
&&B_{11}^{{\cal H}^{\alpha }}\sum_{j=1}^{r}B_{jj}^{{\cal H}^{\alpha
}}-\sum_{j=1}^{r}\left( B_{1j}^{{\cal H}^{\alpha }}\right) ^{2}  \nonumber \\
&\geq &\frac{r-1}{r^{2}}\left( 2\left( \sum_{j=1}^{r}B_{jj}^{{\cal H}%
^{\alpha }}\right) ^{2}-r\sum_{i,j=1}^{r}\left( B_{ij}^{{\cal H}^{\alpha
}}\right) ^{2}\right.   \nonumber \\
&&\left. -\left( r-2\right) \left\vert \sum_{j=1}^{r}B_{jj}^{{\cal H}%
^{\alpha }}\right\vert \sqrt{\frac{r\sum_{i,j=1}^{r}\left( B_{ij}^{{\cal H}%
^{\alpha }}\right) ^{2}-\left( \sum_{j=1}^{r}B_{jj}^{{\cal H}^{\alpha
}}\right) ^{2}}{r-1}}\right\} .  \label{eq-Lemma-RM}
\end{eqnarray}%
Using (\ref{eq-Lemma-RM}) in (\ref{eq-Ric-RM-1}), we obtain 
\begin{eqnarray}
{\rm Ric}_{{\cal H}}^{N_{1}}\left( U_{1}\right)  &\geq &{\rm Ric}_{{\cal R}%
}^{N_{2}}\left( F_{\ast }U_{1}\right) +\frac{r-1}{r^{2}}\left( 2\sum_{\alpha
=r+1}^{n_{2}}\left( \sum_{j=1}^{r}B_{jj}^{{\cal H}^{\alpha }}\right)
^{2}-r\sum_{\alpha =r+1}^{n_{2}}\sum_{i,j=1}^{r}\left( B_{ij}^{{\cal H}%
^{\alpha }}\right) ^{2}\right.   \nonumber \\
&&\left. -\left( r-2\right) \sum_{\alpha =r+1}^{n_{2}}\left\vert
\sum_{j=1}^{r}B_{jj}^{{\cal H}^{\alpha }}\right\vert \sqrt{\frac{%
r\sum_{i,j=1}^{r}\left( B_{ij}^{{\cal H}^{\alpha }}\right) ^{2}-\left(
\sum_{j=1}^{r}B_{jj}^{{\cal H}^{\alpha }}\right) ^{2}}{r-1}}\right\} .
\label{eq-Ric-RM-2}
\end{eqnarray}%
Using Cauchy-Schwartz inequality, we have%
\begin{equation}
\sum_{\alpha =r+1}^{n_{2}}\left\vert \sum_{j=1}^{r}B_{jj}^{{\cal H}^{\alpha
}}\right\vert \sqrt{\frac{r\sum_{i,j=1}^{r}\left( B_{ij}^{{\cal H}^{\alpha
}}\right) ^{2}-\left( \sum_{j=1}^{r}B_{jj}^{{\cal H}^{\alpha }}\right) ^{2}}{%
r-1}}\leq \left\Vert {\rm trace}B^{{\cal H}}\right\Vert \sqrt{\frac{%
r\left\Vert B^{{\cal H}}\right\Vert ^{2}-\left\Vert {\rm trace}B^{{\cal H}%
}\right\Vert ^{2}}{r-1}}.  \label{eq-cauch-squartz-RM}
\end{equation}%
Using (\ref{eq-cauch-squartz-RM}) in (\ref{eq-Ric-RM-2}), we obtain%
\begin{eqnarray*}
{\rm Ric}_{{\cal H}}^{N_{1}}\left( U_{1}\right)  &\geq &{\rm Ric}_{{\cal R}%
}^{N_{2}}\left( F_{\ast }U_{1}\right) +\frac{r-1}{r^{2}}\left( 2\left\Vert 
{\rm trace}B^{{\cal H}}\right\Vert ^{2}-r\left\Vert B^{{\cal H}}\right\Vert
^{2}\right.  \\
&&\left. -\left( r-2\right) \left\Vert {\rm trace}B^{{\cal H}}\right\Vert 
\sqrt{\frac{r\left\Vert B^{{\cal H}}\right\Vert ^{2}-\left\Vert {\rm trace}%
B^{{\cal H}}\right\Vert ^{2}}{r-1}}\right) .
\end{eqnarray*}%
Now, we establish the equality case. Equality in the desired inequality is
attained if and only if equality holds simultaneously in (\ref{eq-Lemma-RM})
and (\ref{eq-cauch-squartz-RM}). If equality holds in (\ref{eq-Lemma-RM}),
then, for each $\alpha \in \{r+1,\ldots ,n_{2}\}$, the matrix $(B_{ij}^{%
{\cal H}^{\alpha }})$ has the form 
\[
B_{ij}^{{\cal H}^{\alpha }}={\rm diag}(a_{\alpha },b_{\alpha },\ldots
,b_{\alpha }).
\]%
On the other hand, equality in (\ref{eq-cauch-squartz-RM}) holds if and only
if the vectors involved in the Cauchy--Schwarz inequality are linearly
dependent. Consequently, there exist constants $\beta _{1},\beta _{2}$, not
both zero, such that 
\[
\beta _{1}\left\vert \sum_{j=1}^{\ell }B_{jj}^{{\cal H}^{\alpha
}}\right\vert +\beta _{2}\sqrt{\frac{\ell \sum_{i,j=1}^{\ell }(B_{ij}^{{\cal %
H}^{\alpha }})^{2}-\left( \sum_{j=1}^{\ell }B_{jj}^{{\cal H}^{\alpha
}}\right) ^{2}}{\ell -1}}=0,\quad \alpha \in \{r+1,\ldots ,n_{2}\}.
\]%
which implies 
\[
\beta _{1}|a_{\alpha }-b_{\alpha }|+\beta _{2}\left( a_{\alpha }+(\ell
-1)b_{\alpha }\right) =0,\quad \alpha \in \{r+1,\ldots ,n_{2}\}.
\]%
This completes the proof.
\end{proof}

\subsection{Simultaneous Chen--Ricci and Hineva inequalities for Riemannian
maps}

In this subsection, we combine Theorems~\ref{Th-Ric-RM} and~\ref{Th-RM-Hin}
to characterize the simultaneous equality cases of the obtained Ricci
curvature inequalities.

\begin{theorem}
\label{Th-sim-Chen-Hin-RM} Let $F:\left( N_{1},g_{1}\right) \rightarrow
\left( N_{2},g_{2}\right) $ be a Riemannian map between Riemannian manifolds
with rank $r<n_{2}$,%
\begin{eqnarray}
&&{\rm Ric}_{{\cal R}}^{N_{2}}\left( F_{\ast }U\right) +\frac{r-1}{r^{2}}%
\left( 2\left\Vert {\rm trace}B^{{\cal H}}\right\Vert ^{2}-r\left\Vert B^{%
{\cal H}}\right\Vert ^{2}\right.  \nonumber \\
&&\left. -\left( r-2\right) \left\Vert {\rm trace}B^{{\cal H}}\right\Vert 
\sqrt{\frac{r\left\Vert B^{{\cal H}}\right\Vert ^{2}-\left\Vert {\rm trace}%
B^{{\cal H}}\right\Vert ^{2}}{r-1}}\right)  \nonumber \\
&\leq &{\rm Ric}_{{\cal H}}^{N_{1}}\left( U\right) \leq {\rm Ric}_{{\cal R}%
}^{N_{2}}\left( F_{\ast }U\right) +\frac{1}{4}\left\Vert {\rm trace}B^{{\cal %
H}}\right\Vert ^{2}.  \label{ineq-bot-side-RM}
\end{eqnarray}%
The equalities in both inequalities of {\rm (\ref{ineq-bot-side-RM})} hold
simultaneously for a unit horizontal vector $U=U_{i}$ if and only if, for
each $\alpha \in \left\{ r+1,\ldots ,n_{2}\right\} $, the matrices take the
following forms 
\[
B^{{\cal H}^{\alpha }}={\rm diag}\left( b_{\alpha },\ldots ,\underset{%
i^{theorem}\ {\rm position}}{\underbrace{a_{\alpha }}},\ldots ,b_{\alpha }\right)
,\quad U=U_{i},\quad i=1,\ldots ,r, 
\]%
and 
\[
\beta _{1}|a_{\alpha }-b_{\alpha }|+\beta _{2}\left( a_{\alpha
}+(r-1)b_{\alpha }\right) =0, 
\]%
where $\beta _{1}$ and $\beta _{2}$ are scalars but not both zero. Moreover, 
\[
a_{\alpha }=\frac{{\rm trace}(B^{{\cal H}^{\alpha }})}{2},\qquad b_{\alpha }=%
\frac{{\rm trace}(B^{{\cal H}^{\alpha }})}{2(r-1)}. 
\]%
Furthermore, the equalities in both inequalities of {\rm (\ref%
{ineq-bot-side-RM})} hold simultaneously for every unit vertical vector $%
U_{i}$, $i=1,\ldots ,r$, if and only if either the fibres of $F$ are totally
geodesic or $r=2$ and 
\[
B_{11}^{{\cal H}^{\alpha }}=B_{22}^{{\cal H}^{\alpha }},\qquad B_{12}^{{\cal %
H}^{\alpha }}=0,\qquad \alpha =r+1,\ldots ,n_{2}. 
\]%
it is straightforward that map $F$ is totally umbilical.
\end{theorem}

\begin{proof} Assume that the equalities in both inequalities of (%
\ref{ineq-bot-side-RM}) hold simultaneously for the unit horizontal vector $%
U=U_{1}$. From the equality case of Theorem~\ref{Th-Ric-RM}, we obtain 
\begin{equation}
B_{1j}^{{\cal H}^{\alpha }}=0,\qquad j=2,\ldots ,r,\quad \alpha =r+1,\ldots
,n_{2},  \label{eq-eq-2-RM}
\end{equation}%
and 
\begin{equation}
B_{11}^{{\cal H}^{\alpha }}=B_{22}^{{\cal H}^{\alpha }}+\cdots +B_{rr}^{%
{\cal H}^{\alpha }},\qquad \alpha =r+1,\ldots ,n_{2}.  \label{eq-eq-3-RM}
\end{equation}%
On the other hand, the equality case of Theorem~\ref{Th-RM-Hin} yields 
\begin{equation}
B_{11}^{{\cal H}^{\alpha }}=a_{\alpha },\qquad B_{22}^{{\cal H}^{\alpha
}}=\cdots =B_{rr}^{{\cal H}^{\alpha }}=b_{\alpha },\qquad \alpha =r+1,\ldots
,n_{2}.  \label{eq-eq-4-RM}
\end{equation}%
Combining (\ref{eq-eq-2-RM}) and (\ref{eq-eq-3-RM}), we obtain%
\[
2B_{11}^{{\cal H}^{\alpha }}={\rm trace}(B^{{\cal H}^{\alpha }}), 
\]%
which immediately gives 
\[
a_{\alpha }=\frac{{\rm trace}(B^{{\cal H}^{\alpha }})}{2}. 
\]%
From (\ref{eq-eq-3-RM}) and (\ref{eq-eq-4-RM}), we obtain 
\[
a_{\alpha }=(r-1)b_{\alpha }. 
\]%
Consequently, 
\[
b_{\alpha }=\frac{{\rm trace}(B^{{\cal H}^{\alpha }})}{2(r-1)}. 
\]%
Suppose that the equalities in (\ref{ineq-bot-side-RM}) hold simultaneously
for every unit horizontal vector $U_{i}$, $i=1,\ldots ,r$. Then, for each $%
\alpha \in \left\{ r+1,\ldots ,n_{2}\right\} $, by the equality case of
Theorem~\ref{Th-Ric-RM}, we obtain 
\begin{equation}
B_{ij}^{{\cal H}^{\alpha }}=0,\qquad i\neq j,\quad i,j\in \left\{ 1,\ldots
,r\right\} ,  \label{eq-RM-sim-1}
\end{equation}%
and 
\begin{equation}
B_{ii}^{{\cal H}^{\alpha }}=\sum_{k=1,k\neq i}^{r}B_{kk}^{{\cal H}^{\alpha
}},\qquad i=1,\ldots ,r.  \label{eq-RM-sim-2}
\end{equation}%
Furthermore, for each $i$, the equality case of Theorem~\ref{Th-RM-Hin}
implies 
\begin{equation}
B_{ii}^{{\cal H}^{\alpha }}=a_{\alpha },\qquad B_{kk}^{{\cal H}^{\alpha
}}=b_{\alpha },\quad k\neq i,\quad k\in \left\{ 1,\ldots ,i-1,i+1,\ldots
,r\right\} .  \label{eq-RM-sim-3}
\end{equation}%
From (\ref{eq-RM-sim-1}) and (\ref{eq-RM-sim-2}), we obtain 
\[
2B_{ii}^{{\cal H}^{\alpha }}=a_{\alpha }+(r-1)b_{\alpha },\qquad i=1,\ldots
,r, 
\]%
which leads to 
\[
2\left( a_{\alpha }+(r-1)b_{\alpha }\right) =r\left( a_{\alpha
}+(r-1)b_{\alpha }\right) . 
\]%
Therefore, 
\[
(r-2)\left( a_{\alpha }+(r-1)b_{\alpha }\right) =0. 
\]%
Thus, 
\[
{\rm either\ }r=2\ {\rm or}\ a_{\alpha }+(r-1)b_{\alpha }=0. 
\]%
If 
\[
a_{\alpha }+(r-1)b_{\alpha }=0, 
\]%
then 
\[
B_{ii}^{{\cal H}^{\alpha }}=0,\qquad B_{ij}^{{\cal H}^{\alpha }}=0\quad
(i\neq j),\quad i,j\in \left\{ 1,\ldots ,r\right\} , 
\]%
which implies that $\nabla F_{\ast }=0$ for all $U_{i}$, $U_{j}\in {\cal H}$%
. Hence, the map $F$ is totally geodesic. If $r=2$, then%
\[
B_{11}^{{\cal H}^{\alpha }}=B_{22}^{{\cal H}^{\alpha }},\qquad B_{12}^{{\cal %
H}^{\alpha }}=0, 
\]%
it is straightforward that map $F$ is totally umbilical.
\end{proof} 

\subsection{Applications to real and complex space forms}

In this subsection, we apply the general result established in the Theorem~%
\ref{Th-sim-Chen-Hin-RM} to Riemannian maps whose target spaces are real and
complex space forms. In particular, we obtain the Chen--Ricci and Hineva
inequalities simultaneously, together with a complete characterization of
their simultaneous equality cases.

\begin{proposition}
Let $F:\left( N_{1},g_{1}\right) \rightarrow \left( N_{2},g_{2}\right) $ be
a Riemannian map a Riemannian manifold to real space form with $\dim
N_{1}=n_{1}$ and $\dim N_{2}=n_{2}$. Then, by using {\rm (\ref%
{Ricci_Curvatures})} and {\rm (\ref{eq-RSF})}, we obtain%
\begin{equation}
{\rm Ric}_{{\cal R}}^{N_{2}}(F_{\ast }U_{1})=c\left( r-1\right) .
\label{eq-Ric-RSF-RM}
\end{equation}
\end{proposition}

\begin{proposition}
Let $F:\left( N_{1},g_{1}\right) \rightarrow \left( N_{2},g_{2}\right) $ be
a Riemannian map a Riemannian manifold to complex space form with $\dim
N_{1}=n_{1}$ and $\dim N_{2}=n_{2}$. Then, by using {\rm (\ref%
{Ricci_Curvatures})} and {\rm (\ref{eq-GCSF})}, we obtain%
\begin{equation}
{\rm Ric}_{{\cal R}}^{N_{2}}(F_{\ast }U_{1})=\frac{c}{4}\left( r-1\right) +%
\frac{3c}{4}\left\Vert PF_{\ast }U_{1}\right\Vert ^{2}.
\label{eq-Ric-CSF-RM}
\end{equation}
\end{proposition}

\begin{theorem}
Let $F:\left( N_{1},g_{1}\right) \rightarrow \left( N_{2},g_{2}\right) $ be
a Riemannian map from a Riemannian manifold to real space form with $\dim
N_{1}=n_{1}$ and $\dim N_{2}=n_{2}$,%
\begin{eqnarray}
&&c\left( r-1\right) +\frac{r-1}{r^{2}}\left( 2\left\Vert {\rm trace}B^{%
{\cal H}}\right\Vert ^{2}-r\left\Vert B^{{\cal H}}\right\Vert ^{2}\right. 
\nonumber \\
&&\left. -\left( r-2\right) \left\Vert {\rm trace}B^{{\cal H}}\right\Vert 
\sqrt{\frac{r\left\Vert B^{{\cal H}}\right\Vert ^{2}-\left\Vert {\rm trace}%
B^{{\cal H}}\right\Vert ^{2}}{r-1}}\right)  \nonumber \\
&\leq &{\rm Ric}_{{\cal H}}^{N_{1}}\left( U\right) \leq c\left( r-1\right) +%
\frac{1}{4}\left\Vert {\rm trace}B^{{\cal H}}\right\Vert ^{2}.
\label{eq-RM-RSFboth}
\end{eqnarray}
\end{theorem}

\begin{proof} Using (\ref{eq-Ric-RSF-RM}) in (\ref{ineq-bot-side-RM}%
), we obtain (\ref{eq-RM-RSFboth}).
\end{proof} 

\begin{theorem}
Let $F:\left( N_{1},g_{1}\right) \rightarrow \left( N_{2},g_{2}\right) $ be
a Riemannian map from a Riemannian manifold to complex space form with $\dim
N_{1}=n_{1}$ and $\dim N_{2}=n_{2}$,%
\begin{eqnarray}
&&\frac{c}{4}\left( r-1\right) +\frac{3c}{4}\left\Vert PF_{\ast
}U_{1}\right\Vert ^{2}+\frac{r-1}{r^{2}}\left( 2\left\Vert {\rm trace}B^{%
{\cal H}}\right\Vert ^{2}-r\left\Vert B^{{\cal H}}\right\Vert ^{2}\right. 
\nonumber \\
&&\left. -\left( r-2\right) \left\Vert {\rm trace}B^{{\cal H}}\right\Vert 
\sqrt{\frac{r\left\Vert B^{{\cal H}}\right\Vert ^{2}-\left\Vert {\rm trace}%
B^{{\cal H}}\right\Vert ^{2}}{r-1}}\right)  \nonumber \\
&\leq &{\rm Ric}_{{\cal H}}^{N_{1}}\left( U\right) \leq \frac{c}{4}\left(
r-1\right) +\frac{3c}{4}\left\Vert PF_{\ast }U_{1}\right\Vert ^{2}+\frac{1}{4%
}\left\Vert {\rm trace}B^{{\cal H}}\right\Vert ^{2}.  \label{eq-RM-CSF-both}
\end{eqnarray}
\end{theorem}

\begin{proof}Using (\ref{eq-Ric-CSF-RM}) in (\ref{ineq-bot-side-RM}%
), we obtain (\ref{eq-RM-CSF-both}).
\end{proof}

\section{Future Research Directions}\label{sec_5}
The results obtained in this paper also open several directions for future
research. One natural direction is to extend the present two-sided
Chen--Ricci and Hineva inequalities to Riemannian submersions associated
with richer ambient structures, including quaternionic space forms,
Sasakian space forms, Kenmotsu space forms, cosymplectic space forms,
generalized Sasakian space forms, and other almost Hermitian and almost
contact metric space forms. It is also natural to investigate these
inequalities for special classes of Riemannian submersions, such as
invariant, anti-invariant, semi-invariant, slant, semi-slant, hemi-slant,
and bi-slant submersions.

A parallel research direction concerns Riemannian maps. Since Riemannian
maps generalize both Riemannian immersions and Riemannian submersions,
the present results provide a starting point for studying two-sided
Ricci-curvature inequalities for invariant, anti-invariant,
semi-invariant, slant, semi-slant, hemi-slant, bi-slant, and other
structure-preserving Riemannian maps. In the case of complex space forms,
it is particularly interesting to investigate holomorphic,
anti-holomorphic, invariant, anti-invariant, slant, and semi-slant
Riemannian maps. Another important problem is the rigidity and
classification of the equality cases, especially the geometric
characterization of maps and submersions for which the Chen--Ricci and
Hineva inequalities are simultaneously attained. Finally, extending the
direct Ricci-curvature approach developed here to statistical submersions,
statistical maps, and other generalized geometric settings may lead to a
more unified theory of two-sided Ricci-curvature inequalities.

\medskip

\noindent{\bf Acknowledgements:} The author gratefully acknowledges the Human Resource Development Group (HRDG), Council of Scientific and Industrial Research (CSIR), New Delhi, India, for awarding the CSIR Junior Research Fellowship (File No. 09/0013(16054)/2022-EMR-I).

\section*{Declarations}
	
\begin{itemize}	
\item Conflict of interest/Competing interests, Use of generative AI: Not applicable.
		
\item Ethics approval: The work is original and is not submitted to other journals simultaneously.
		
\item Availability of code, data, and materials: Not available.
\end{itemize}

\end{document}